\documentclass[11pt]{article}

\usepackage{amsmath}
\usepackage{color}
\usepackage{amssymb}
\usepackage{amsfonts}
\usepackage{amsthm}
\usepackage{amscd}
\usepackage{latexsym}
\usepackage[pdftex]{hyperref} 

\newcommand{\enveq}[1]{\begin{equation}#1\end{equation}}
\newcommand{\enveqn}[1]{\begin{equation*}#1\end{equation*}}

\newcommand{\envgan}[1]{\begin{gather*}#1\end{gather*}}
\newcommand{\enval}[1]{\begin{align}#1\end{align}}
\newcommand{\envaln}[1]{\begin{align*}#1\end{align*}}

\newcommand{\bma}{\left(\begin{array}{cc}}
\newcommand{\ema}{\end{array}\right)}
\newcommand{\bca}{\left(\begin{array}{c}}
\newcommand{\eca}{\end{array}\right)}

\newcommand{\cA}{\ensuremath{\mathcal{A}}}
\newcommand{\J}{\mathcal{J}}

\newcommand{\cD}{\ensuremath{\mathcal{D}}}
\newcommand{\domain}{\ensuremath{\mathrm{Dom}}}
\newcommand{\cH}{\ensuremath{\mathcal{H}}}

\newcommand{\cO}{\ensuremath{\mathcal{O}}}

\newcommand{\bR}{\ensuremath{\mathbb{R}}}
\newcommand{\bC}{\ensuremath{\mathbb{C}}}
\newcommand{\bZ}{\ensuremath{\mathbb{Z}}}
\newcommand{\bN}{\ensuremath{\mathbb{N}}}

\newcommand{\la}{\left\langle}
\newcommand{\ra}{\right\rangle}

\newcommand{\ox}{\otimes}

\newcommand{\half}{\frac{1}{2}}
\newcommand{\thalf}{\tfrac{1}{2}}

\newcommand{\p}{\partial}
\newcommand{\pe}{\ensuremath{\p_{e}}}
\newcommand{\pf}{\ensuremath{\p_{f}}}
\newcommand{\pk}{\ensuremath{\p_{k}}}
\newcommand{\pki}{\ensuremath{\p_{k}^{-1}}}

\newcommand{\re}[3]{t_{#2, #3}^{#1}}

\newcommand{\coms}{\tilde{S}}
\newcommand{\comt}{\tilde{T}}

\advance\topmargin by -\headheight
\advance\topmargin by -\headsep
\evensidemargin=\oddsidemargin
\makeatletter
\def\section{\@startsection{section}{1}{\z@}{-3.5ex plus -1ex minus
  -.2ex}{2.3ex plus .2ex}{\large\bf}}
\def\subsection{\@startsection{subsection}{2}{\z@}{-3.25ex plus -1ex
  minus -.2ex}{1.5ex plus .2ex}{\normalsize\bf}}
\makeatother

\numberwithin{equation}{section} 

\theoremstyle{plain} 
\newtheorem{thm}{Theorem}[section]
\newtheorem{lemma}[thm]{Lemma}
\newtheorem{prop}[thm]{Proposition}
\newtheorem{corl}[thm]{Corollary}

\theoremstyle{definition} 

\newcommand{\A}{\mathcal{A}}  
\newcommand{\B}{\mathcal{B}}  
\newcommand{\C}{\mathbb{C}}   
\newcommand{\D}{\mathcal{D}}  
\renewcommand{\H}{\mathcal{H}}  
\newcommand{\M}{\mathcal{M}}  
\newcommand{\N}{\mathbb{N}}   
\newcommand{\R}{\mathbb{R}}   
\newcommand{\T}{\mathbb{T}}   

\title{A residue formula for the fundamental Hochschild $3$-cocycle for $SU_{q}(2)$}

\author{Ulrich
Kr\"{a}hmer\dag\thanks{email:
\texttt{ulrich.kraehmer@glasgow.ac.uk}, 
\texttt{adam.rennie@anu.edu.au},
\texttt{roger.senior@anu.edu.au}},\ 
\ Adam Rennie\ddag,\ 
\ Roger Senior\ddag \\[6pt]
\dag School of Mathematics \&
Statistics, University of Glasgow\\
15 University Gardens, G12 8QW Glasgow, Scotland\\[6pt]
\ddag Mathematical Sciences Institute,
Australian National University\\
Acton, ACT, 0200, Australia\\[6pt]
}

\begin{document}

\maketitle

\vspace{-8pt}

\begin{abstract}
An analogue of a spectral triple over $SU_{q}(2)$ is constructed for which the usual assumption 
of bounded commutators with the Dirac operator fails. 
An analytic expression analogous to that for the Hochschild
class of the Chern character for spectral triples 
yields a non-trivial twisted Hochschild 3-cocycle. 
The problems arising from the unbounded commutators 
are overcome by defining a residue functional using projections to cut
down the Hilbert space. 
\end{abstract}

\section{Introduction}
\label{sec:intro}

This paper studies the homological dimension of the quantum group 
$SU_q(2)$ from the perspective of Connes' spectral triples.
We use an analogue of a spectral triple to construct, by a
residue formula, a nontrivial Hochschild 3-cocycle. 
Thus we obtain  finer dimension information
than is provided by the nontriviality of a 
$K$-homology class, which is sensitive only to dimension modulo 2.

The position of quantum groups 
within noncommutative geometry has been studied intensively
over the last 15 years. 
In particular, Chakraborty
and Pal \cite{ChP1} introduced a spectral triple for $SU_q(2)$, and this construction was subsequently refined
in \cite{DLSSV} and generalised by 
Neshveyev and Tuset in \cite{NT2} to all compact Lie groups $G$.
These spectral triples have analytic dimension $\dim G$ 
and nontrivial $K$-homology class. 
However, when Connes computed the Chern character for Chakraborty and
Pal's spectral triple \cite{C1}, 
he found that it had cohomological dimension $1$ in the sense that 
the degree $\dim SU(2)=3$ term
in the local index formula is a Hochschild coboundary.  
Analogous results for the spectral triple from
\cite{DLSSV} were obtained in \cite{DLSSV2}.

Contrasting these `dimension drop' results, 
Hadfield and the first author \cite{HK1,HK2} 
showed that $SU_q(2)$ is a twisted Calabi-Yau algebra of dimension 3 whose twist is the inverse of the modular automorphism for the Haar state on this compact quantum group, cf. Section \ref{sec:suq2}. 
They also computed a cocycle 
representing a generator of the nontrivial degree $3$ 
Hochschild cohomology groups (which we call the fundamental cocycle), and 
a dual degree $3$ Hochschild cycle  
which we denote $dvol$.

The starting point of the present paper is 
the concept of a `modular' spectral triple \cite{CNNR}. These are analogous to
ordinary spectral triples except for the use of twisted traces. 
The examples considered in \cite{CNNR} arise from
KMS states of circle actions on $C^*$-algebras, and yield nontrivial
$KK$-classes with $1$-dimensional Chern characters 
in twisted cyclic cohomology.
In \cite{KW} it was then shown that they also can be used to obtain
the fundamental cocycle of the standard Podle\'s quantum 2-sphere.

Motivated by this, our
construction here extends the modular spectral triple on 
the Podle\'s sphere to all of $SU_q(2)$. This extension is not a
modular spectral triple, 
but as our main theorem shows, still
captures the homological dimension 3:
we give a residue formula for a twisted Hochschild $3$-cocycle which
is a nonzero multiple of the fundamental cocycle. We obtain this formula
by analogy with Connes' formula for the Hochschild class of the Chern
character of spectral triples, 
\cite[Theorem 8, IV.2.$\gamma$]{C} and \cite{BeF,CPRS1}.  A natural next question that
arises is whether our constructions provide a representative of a
nontrivial $K$-homological class.

The organisation of the paper is as follows. In Section  \ref{sec:suq2} we recall the definitions of 
$SU_q(2)$, the Haar state on $SU_q(2)$ and the associated GNS representation, and finally the modular theory
of the Haar state. In Section \ref{sec:homology} we recall the homological constructions of \cite{HK1,HK2}, 
and prove some elementary results 
we will need when we come to show that our residue cocycle does indeed
recover the class of the fundamental cocycle.

Section \ref{sec:mero} contains all the key analytic results 
on meromorphic extensions of certain functions 
that allow us to prove novel summability type results for operators whose 
eigenvalues have mixed polynomial and exponential growth,
see Lemma \ref{lem:general-pole-3}.

Section \ref{sec:spec} constructs an analogue of a spectral triple $(\A,\H,\D)$
over the algebra $\A$ of polynomials in the 
standard generators
of the $C^*$-algebra $SU_q(2)$. The key requirement of
bounded commutators fails, and this `spectral triple' fails to be finitely summable in the usual sense 
(however, it is $\theta$-summable). Using an ultraviolet cutoff 
we can recover finite summability of the operator $\D$ on a subspace
of $\H$ with respect to a suitable twisted trace. 
However, our representation of $\A$ does not restrict to this 
subspace, and so we are prevented from obtaining a genuine spectral triple. 


In Section \ref{sec:res-hochs} we define a residue functional $\tau$. Heuristically, for an operator $T$, 
$$
\tau(T)=\mbox{Res}_{s=3}\mbox{Trace}(\Delta^{-1} QT(1+\D^2)^{-s/2}).
$$
Here $\Delta$ implements the modular 
automorphism of the Haar state, $\D$ is  our Dirac operator and $Q$ is
a suitable projection that implements the cutoff. The existence, first
of the trace, and then the residue, are both nontrivial matters. 

The main properties of $\tau$ are described in 
Theorem  \ref{thm_res_gives_int}, and in particular we show that
the domain of $\tau$ contains the products of commutators $a_{0} [\cD, a_{1}] [\cD, a_{2}] [\cD, a_{3}]$ for
$a_i\in\A$. In addition, $\tau$ is a twisted trace on a suitable subalgebra of the domain containing these products. 
The main result, Theorem \ref{thm:big-fish}, proves that the map 
$a_{0}, \ldots, a_{3} \mapsto \tau(a_{0} [\cD, a_{1}] [\cD, a_{2}] [\cD, a_{3}])$ is a twisted Hochschild
3-cocycle, whose cohomology class is non-trivial and coincides with (a multiple of) the fundamental class.

\medskip

{\bf Acknowledgements} We would like to thank our colleagues Alan
Carey, Victor Gayral, Jens Kaad, Andrzej Sitarz and Joe V\'arilly for stimulating  
discussions on these topics. 
The second and third authors were supported by the Australian Research Council. 
The first author was supported by 
the EPSRC fellowship EP/E/043267/1 and partially by the
Polish Government Grant N201 1770 33 and 
the Marie Curie PIRSES-GA-2008-230836
network.

\section{Background on $SU_{q}(2)$}
\label{sec:suq2}

The notations and conventions of \cite{KS} will be used throughout for consistency. We 
recall that $\A:=\cO(SU_q(2))$, for $q \in (0, 1)$, is the unital Hopf $*$-algebra with generators 
$a, b, c, d$ satisfying the relations

\envaln{
ab = qba, \ \ \ ac = qca, \ \ \ bd &= qdb, \ \ \ cd = qdc, \ \ \ bc = cb \\
ad = 1 + qbc, \ & \ \ da = 1 + q^{-1}bc
}

and carrying the usual Hopf structure,
as in e.g. \cite{KS}. The involution is
given by 

\enveqn{
a^{\ast} = d, \ \ \ b^{\ast} = -qc, \ \ \ c^{\ast} = -q^{-1}b, \ \ \ d^{\ast} = a.
}

We choose to view $\A$ as being generated by $a, b, c, d$ explicitly, 
rather than just $a, b$, in order to make formulae more readable. 

\begin{prop}[{\cite[Proposition 4.4]{KS}}]
The set $\{ a^{n}b^{m}c^{r}, \  b^{m}c^{r}d^{s} \ | \ m, r, s \in \bN_{0}, \ n \in \bN \}$ 
is a vector space basis of $\A$. These monomials will be referred to as the polynomial basis.
\end{prop}

Recall that for each $l \in \half \bN_0$, there is a 
unique (up to unitary equivalence) irreducible corepresentation $V_l$ of 
the coalgebra $\A$ of dimension $2l+1$, and that $\A$ is
cosemisimple. That is, if we fix a vector space basis in
each of the $V_l$ and denote by $t^l_{i,j} \in \A$ the corresponding
matrix coefficients, then we have the following analogue of the
Peter-Weyl theorem.


\begin{thm}[{\cite[Theorem 4.13]{KS}}] \label{thm_rep_basis}
Let 
$ I_{l} := \{ -l, -l+1, \ldots, l-1,l \} $.
Then the set $\{ \re{l}{i}{j} \ | \ l \in
 \half \bN_{0}, \ i, j \in I_{l} \}$ is a vector space basis of $\A$.
\end{thm}

This will be referred to as the
Peter-Weyl basis. With a suitable choice of basis in
$V_\half$, one has
\envaln{
a &= \re{\half}{-\half}{-\half}, & b &= \re{\half}{-\half}{\half}, & c &= \re{\half}{\half}{-\half}, 
& d &= \re{\half}{\half}{\half}.
}
The expressions
for the Peter-Weyl basis elements as
linear combinations of the polynomial 
basis elements can be found in
\cite[Section 4.2.4]{KS}.

The quantized universal enveloping algebra $U_{q}(\mathfrak{sl}(2))$ is a 
Hopf algebra which is generated by $k, k^{-1}, e, f$ with relations

\enveqn{
kk^{-1} = k^{-1}k = 1, \quad kek^{-1} = qe, \quad kfk^{-1} = q^{-1}f, 
\quad [e, f] = \frac{k^{2} - k^{-2}}{q - q^{-1}}.
}

Note that in \cite{KS} this algebra is
denoted by 
$\breve{U}_{q}(\mathrm{sl}_{2})$ 
and $U^\mathrm{ext}_{q}(\mathrm{sl}_{2})$. The algebra
$U_{q}(\mathfrak{sl}(2))$ carries the 
following Hopf structure

\envaln{
\Delta(k) = k \ox k, \quad \Delta(e) = e \ox k &+ k^{-1} \ox e, \quad \Delta(f) = f \ox k + k^{-1} \ox f \\
S(k) = k^{-1}, \quad S(e) &= -qe, \quad S(f) = -q^{-1}f \\
\varepsilon(k) = 1, \quad & \varepsilon(e) = \varepsilon(f) = 0.
}

Adding the following involution

\enveqn{
k^{\ast} = k, \quad e^{\ast} = f, \quad f^{\ast} = e
}

we obtain a Hopf $\ast$-algebra which 
we denote by $U_{q}(\mathfrak{su}(2))$. 

\begin{thm}[{\cite[Theorem 4.21]{KS}}]
\label{theorem_dual_pairing}
There exists a unique dual pairing $\la \cdot , \cdot \ra$ of the 
Hopf algebras $U_{q}(\mathfrak{sl}(2))$ and $\A$ such that

\envaln{
\la k, a \ra = q^{-\half}, &\quad \la k, d \ra = q^{\half}, \quad \la e, c \ra = \la f, b \ra = 1 \\
\la k, b \ra = \la k, c \ra = \la e, a \ra = &\la e, b \ra = \la e, d \ra = \la f, a \ra = \la f, c \ra = \la f, d \ra = 0.
}
This pairing is compatible with the $*$-structures on $U_{q}(\mathfrak{sl}(2))$ and $\A$, \cite[Chapter 1]{KS}.
\end{thm}

The dual pairing between the Hopf algebras 
$\la \cdot, \cdot \ra \colon U_{q}(\mathfrak{sl}(2)) \times \A \rightarrow \bC$ 
defines left and right actions of $U_{q}(\mathfrak{sl}(2))$ on $\A$. 
Using Sweedler notation ($\Delta(x) = \sum x_{(1)} \ox x_{(2)}$) these actions are given by

\envaln{
g \triangleright x &:= \sum x_{(1)} \la g, x_{(2)} \ra & x \triangleleft g &:= \sum x_{(2)} \la g, x_{(1)} \ra, & &
\text{for all} \ \ x \in \A, \ g \in U_{q}(\mathfrak{sl}(2)).
}

The left and right actions make $\A$ a $U_{q}(\mathfrak{sl}(2))$-bimodule 
\cite[Proposition 1.16]{KS}. 

Our definition of  the $q$-numbers is

\enveqn{
[a]_{q} := \frac{q^{-a} - q^{a}}{q^{-1} - q}=Q(q^{-a} - q^{a}) \qquad \text{for any} \ a \in \bC,
}

where we abbreviated 
$Q := (q^{-1} - q)^{-1} \in (0, \infty)$.
The following lemma recalls the explicit formulas for the action of the generators on the
Peter-Weyl basis.

\begin{lemma}
For all $n \in \bZ$,

\envaln{
k^{n} \triangleright \re{l}{i}{j} &= q^{nj} \re{l}{i}{j} & \re{l}{i}{j} \triangleleft k^{n} &= q^{ni} \re{l}{i}{j} \\
e \triangleright \re{l}{i}{j} &=  \sqrt{\left[ l + \thalf \right]_{q}^{2} - \left[ j + \thalf \right]_{q}^{2} } \,\,\re{l}{i}{j+1} & f \triangleright \re{l}{i}{j} &= \sqrt{\left[ l + \thalf \right]_{q}^{2} - \left[ j - \thalf \right]_{q}^{2} } \,\, \re{l}{i}{j-1}.
}
%

\end{lemma}

Later we will use the notation

\enveqn{
\pk := k \triangleright \cdot\,, \qquad \pe := e \triangleright \cdot\,, \qquad \pf := f \triangleright \cdot\,,
}

especially when we extend these
operators from $\A$ to
suitable completions.
Also observe that 
$\Delta(k^{n}) = k^{n} \ox k^{n}$ for all $n \in \bZ$, hence $k^{n} \triangleright \cdot$ 
and $\cdot \triangleleft k^{n}$ are
algebra automorphisms on
$\A$. They are not $*$-algebra automorphisms since for $\alpha\in\A$ we have
$(k \triangleright \alpha)^*=k^{-1} \triangleright
\alpha^*,\ (\alpha \triangleleft k)^*=\alpha^*
\triangleleft k^{-1}$. Finally, we 
introduce
$$
\partial_{H}(\re{l}{i}{j}) =  j \re{l}{i}{j},
$$
and we note that formally $\partial_k=q^{\partial_H}$.

\subsection{The GNS representation for the Haar state}

We denote by $A:=C^*(SU_q(2))$ the
universal $C^*$-completion of the
$*$-algebra $\A$ \cite[Section~4.3.4]{KS}.
Let $h$ be the Haar state
of $A$ whose values on basis
elements are

\enveqn{
 h(a^{i}b^{j}c^{k}) = h(d^{i}b^{j}c^{k}) = \delta_{i, 0} \delta_{j, k} (-1)^{k} [k+1]_q^{-1}, \quad
h(\re{l}{i}{j}) = \delta_{l0}.
}

Let $\cH_{h}$ denote the GNS space $L^{2}(A, h)$, 
where the inner product $\la x, y \ra = h(x^{\ast} y)$ is conjugate linear in the first variable. 
The representation of $A$ on $\cH_{h}$ is is induced by left
multiplication in $A$. 
The set $\{ \re{l}{i}{j} \ | \ l \in
\half \bN_{0}, \ i, j \in I_{l} \}$ is
an orthogonal basis for $\H_h$ with

\enveqn{
\la \re{l}{i}{j}, \re{l'}{i'}{j'} \ra = \delta_{l,l'} \delta_{i,i'} \delta_{j,j'} q^{-2i} [2l+1]_{q}^{-1}.
}


\subsection{Modular Theory}
Following Woronowicz, we call the automorphism
$$
		  \vartheta (\alpha) := k^{-2}
		  \triangleright \alpha \triangleleft
		  k^{-2},\quad
		  \alpha \in \A 
$$
the modular automorphism of $\A$.
The action of $\vartheta$ on the
generators of $\A$ and the Peter-Weyl 
basis  is given by

\enveqn{
\vartheta(a) = q^{2}a, \ \ \vartheta(b) = b, \ \ \vartheta(c) = c, \ \ \vartheta(d) = q^{-2}d, \ \ 
\vartheta(\re{l}{r}{s}) = q^{-2(r+s)} \re{l}{r}{s}.
}

The modular automorphism is a (non $\ast$-) algebra automorphism; more precisely for any 
$\alpha \in \A$

\enveqn{
\vartheta(\alpha)^{\ast} = \vartheta^{-1}(\alpha^{\ast}).
}

The Haar state is related to the modular automorphism by the following proposition.

\begin{prop}[{\cite[Proposition 4.15]{KS}}]
\label{prop_haar_twisted}
For $\alpha,\,\beta \in \A$, we have $h(\alpha\beta) = h(\vartheta(\beta)\alpha)$.
\end{prop}





In fact, $h$ extends to a KMS state on
$A$ for the strongly
continuous one-parameter group
$\vartheta_{t}$, $t \in \mathbb{R}$, of 
$*$-automorphisms of
$A$ 
which is given on the generators by

\envaln{
\vartheta_{t}(a) &:= q^{-2it} a\,, &
\vartheta_{t}(b) &:= b\,, &
\vartheta_{t}(c) &:= c\,, &
\vartheta_{t}(d) &:= q^{2it} d\,.
}

We extend this to an 
action $\vartheta_{\cdot}\, \colon \bC \times \A \rightarrow \A$ 
by algebra (not $\ast$-) automorphisms that is defined on generators
by
\envaln{
\vartheta_{z}(a) &:= q^{-2iz} a\,, &
\vartheta_{z}(b) &:= b\,, &
\vartheta_{z}(c) &:= c\,, &
\vartheta_{z}(d) &:= q^{2iz} d\,,
}

so that the modular automorphism
$\vartheta$ is $\vartheta_i$.

%



We can implement $\vartheta_t$ in
the GNS representation on $\cH_h$.
To do this, we define an unbounded
linear operator $\Delta_{F}$
on $\A\subset \cH_h$  by

\enveqn{
\Delta_{F}(\re{l}{i}{j}) := q^{2i+2j} \re{l}{i}{j}
}

and call this the full modular
operator. Then we have

\enveqn{
\vartheta_{t}(x) \xi = \Delta_F^{it} x
\Delta_F^{-it} \xi\,, \qquad \text{for all}
\ x \in A \ \
\text{and} \ \xi \in \cH_h. 
}

The subscript $F$ denotes that this operator is 
associated to the full modular
automorphism $\vartheta$. In addition,
we define the left and the right
modular operators on $\A\subset \H_h$ by

\envaln{
\Delta_{L}(\re{l}{i}{j}) &:= q^{2j} \re{l}{i}{j}, & \Delta_{R}(\re{l}{i}{j}) &:= q^{2i} \re{l}{i}{j},
}

so  $\Delta_{F} =
\Delta_{L} \Delta_{R} = \Delta_{R}
\Delta_{L}$. Just as $\Delta_F$
implements the modular automorphism
group, the left and right modular
operators implement one-parameter groups
of automorphisms of $A$:

\envaln{
& \sigma_{L, t}(\re{l}{r}{s}) = q^{2its} \re{l}{r}{s} = \Delta_{L}^{it} \re{l}{r}{s} \Delta_{L}^{-it}\,, & & \sigma_{R, t}(\re{l}{r}{s}) = q^{2itr} \re{l}{r}{s} = \Delta_{R}^{it} \re{l}{r}{s} \Delta_{R}^{-it}.
}

As with the full action, the left and
right actions are periodic and hence
give rise to actions of $\T$ on $A$.
These
may be extended to a complex action 
on the $\ast$-subalgebra $\A$ which we will 
denote $\sigma_{L, z}$ and $\sigma_{R,
z}$. In particular, we obtain for $z=i$
the algebra automorphisms

\envaln{
\sigma_{L} &:= k^{-2} \triangleright \cdot & \sigma_{R} &:= \cdot \triangleleft k^{-2} & \vartheta &= \sigma_{L} \sigma_{R} = \sigma_{R} \sigma_{L} \\
& & \sigma_{L}(\re{l}{r}{s}) = q^{-2s} \re{l}{r}{s} \qquad & \qquad \sigma_{R}(\re{l}{r}{s}) = q^{-2r} \re{l}{r}{s}\\
\vartheta(\alpha) \xi &= \Delta_{F}^{-1} \alpha \Delta_{F} \xi &
\sigma_{L}(\alpha) \xi &= \Delta_{L}^{-1} \alpha \Delta_{L} \xi &
\sigma_{R}(\alpha) \xi &= \Delta_{R}^{-1} \alpha \Delta_{R} \xi.
}

The fixed point algebra for the left
action on $\A$ is isomorphic to the 
standard Podle\'s quantum 2--sphere
$\cO(S_{q}^{2})$. 
We will denote its $C^{\ast}$-completion
by $B$.
As the left action is periodic, we may define a positive faithful expectation 
$\Phi \colon A \rightarrow B$ by

\enveqn{
\Phi(x) = \frac{\ln(q^{-2})}{2\pi } \int_{0}^{2\pi/\ln (q^{-2})}
\sigma_{L, t}(x) dt.
}

More generally, given $n \in \bZ$ and $x \in A$ we define

\enveqn{
\Phi_{n}(x) = \frac{\ln(q^{-2})}{2\pi } \int_{0}^{2\pi/\ln(q^{-2})} t^{-n} \sigma_{L, t}(x) dt.
}

Since $\sigma_{L, t}$ is a strongly
continuous action on $A$, the $\Phi_{n}$ are 
continuous maps on $A$. Observe that $\Phi = \Phi_{0}$ and

\enveqn{
\Phi_{n}(\re{l}{i}{j}) = \delta_{n, 2j} \re{l}{i}{j}
}

Hence the $\Phi_{n}$ can be extended to bounded operators on the GNS space $\cH_{h}$, 
and in fact the $\Phi_{n}$ are projections onto the spectral subspaces of the left circle action. 
So we make explicit the decomposition of $A$ into the left spectral subspaces by defining

\envaln{
		  B_{n} &:= \Phi_{n}(A) = \{ \alpha \in A \ | \ \sigma_{L, t}(\alpha) = q^{2int} \alpha \} \quad\mbox{and}\quad
\cH_{n} := L^{2}(B_{n}, h)
}

where $h$ is the Haar state (restricted to $B_{n}$). This leads to the following decomposition 
for the GNS space

\enveqn{
\cH_{h} = \bigoplus_{n = -\infty}^{\infty} \cH_{n}.
}

The commutation relations for the projections $\Phi_{n}$ and the operators 
$\pk$, $\pe$ and $\pf$ are found from the definitions on the Peter-Weyl basis to be

\envaln{
\pk \Phi_{n} &= \Phi_{n} \pk =
q^{\frac{n}{2}} \Phi_{n} 
& \partial_{H} \Phi_{n} &= \Phi_{n} \partial_{H} = \frac{n}{2} \Phi_{n}
& \Delta_{L} \Phi_{n} &= \Phi_{n} \Delta_{L} = q^{n} \Phi_{n} \\
\pe \Phi_{n} &= \Phi_{n+2} \pe & \pf \Phi_{n} &= \Phi_{n-2} \pf\,.
}

The left actions of $e$ and $f$ are twisted derivations in the sense that for $\alpha,\,\beta \in \A$

\envaln{
\pe(\alpha\beta) &= \pe(\alpha) \pk(\beta) + \pki(\alpha) \pe(\beta) \\
\pf(\alpha\beta) &= \pf(\alpha) \pk(\beta) + \pki(\alpha) \pf(\beta)\,.
}

More generally, given $\alpha \in A$ and $\xi \in \cH_{h}$

\enval{
\pe(\alpha \xi) &= \pe(\alpha) \Delta_{L}^{\half} \xi + \sigma_{L}^{\half}(\alpha) \pe(\xi) &
\pf(\alpha \xi) &= \pf(\alpha) \Delta_{L}^{\half} \xi + \sigma_{L}^{\half}(\alpha) \pf(\xi) \label{eqn_twisted_ef} \\
&= \pe(\alpha) \Delta_{L}^{\half} \xi + \Delta_{L}^{-\half} \alpha \Delta_{L}^{\half} \pe(\xi) &
&= \pf(\alpha) \Delta_{L}^{\half} \xi + \Delta_{L}^{-\half} \alpha \Delta_{L}^{\half} \pf(\xi). \nonumber
}

See e.g.~\cite{scream} and the references therein 
for background on the generalisation of this setting in terms of
Hopf-Galois extensions.

\section{Twisted homology and cohomology}
\label{sec:homology}
We recall that the algebra $\A$ is a $\vartheta^{-1}$-twisted
Calabi-Yau algebra of dimension 3,
see~\cite{HK2} and the references therein
for this result and some background. Since the centre of $\A$
consists only of the scalar multiples of
$1_\A$, this means 
in particular that the
cochain complex 
$C^\bullet:=\mathrm{Hom}_\mathbb{C}
(\A^{\otimes_\mathbb{C}
\bullet+1},\mathbb{C})$, with differential
$b_{\vartheta^{-1}} : C^n \rightarrow C^{n+1}$
given by
\envaln{
(b_{\vartheta^{-1}}\varphi)(a_{0}, \ldots, a_{n}, a_{n+1}) 
&= \sum_{i = 0}^{n} (-1)^{n} \varphi(a_{0}, \ldots, a_{i}a_{i+1}, \ldots, a_{n+1}) \\
& \quad + (-1)^{n+1} \varphi(\vartheta^{-1}(a_{n+1})a_{0}, a_{1}, \ldots, a_{n}),
}

is exact in degrees $n>3$ and has third cohomology 
$H^3(C,b_{\vartheta^{-1}}) \simeq \mathbb{C}$.
An explicit cocycle whose cohomology class generates
$H^3(C,b_{\vartheta^{-1}})$ can be
constructed using the following
incarnation of the cup product $\smallsmile$ in
Hochschild cohomology: 

\begin{lemma}\label{tuna}
Let 
$\sigma_0,\ldots,\sigma_3$
be automorphisms of $\A$,
$\int : \A \rightarrow \mathbb{C}$ be a
$\sigma_0 \circ \vartheta^{-1} \circ \sigma_3^{-1}$-twisted trace, that is, 
$$
		  \int \alpha\beta=\int \sigma_0(\vartheta^{-1}(\sigma_3^{-1}(\beta)))\alpha,
$$
and $\partial_i : \A \rightarrow \A$, $i=1,2,3$, 
be $\sigma_{i-1}$-$\sigma_i$-twisted
derivations, that is,
$$
		  \partial_i(\alpha\beta)=\sigma_{i-1}(\alpha)
		  \partial_i (\beta)+
		  \partial_i(\alpha) \sigma_i(\beta).
$$ 
Then the functional defined via the cup product by
$$
		  \left(\int \smallsmile \partial_1
		  \smallsmile \partial_2
		  \smallsmile
		  \partial_3\right)(a_0,a_1,a_2,a_3):=
		  \int \sigma_0(a_0) \partial_1(a_1)
		  \partial_2(a_2) \partial_3(a_3)
$$
is a $\vartheta^{-1}$-twisted cocycle, $b_{\vartheta^{-1}} (\int \smallsmile \partial_1
		  \smallsmile \partial_2
		  \smallsmile
		  \partial_3)=0$.
\end{lemma}

\begin{proof}
This is a straightforward computation:
\begin{align*}
&\quad \left(b_{{\vartheta^{-1}}}\int \smallsmile \partial_1
		  \smallsmile \partial_2
		  \smallsmile
		  \partial_3\right)(a_0,a_1,a_2,a_3,a_4)\\
&=\ \int \sigma_0(a_0a_1) \partial_1(a_2)
		  \partial_2(a_3) \partial_3(a_4)
		  -\int \sigma_0(a_0) \partial_1(a_1a_2)
		  \partial_2(a_3) \partial_3(a_4)\\
&\quad\ +\int \sigma_0(a_0) \partial_1(a_1)
		  \partial_2(a_2a_3)
 \partial_3(a_4)
		  -\int \sigma_0(a_0) \partial_1(a_1)
		  \partial_2(a_2) \partial_3(a_3a_4)\\
&\quad\ +\int \sigma_0({{\vartheta^{-1}}}(a_4)a_0) \partial_1(a_1)
		  \partial_2(a_2) \partial_3(a_3)\\
&=\ -\int \sigma_0(a_0) \partial_1(a_1)
		  \partial_2(a_2) \partial_3(a_3)
 \sigma_3(a_4)
		  +\int \sigma_0({{\vartheta^{-1}}}(a_4)) \sigma_0(a_0) \partial_1(a_1)
		  \partial_2(a_2) \partial_3(a_3)\\
&=\  0.\qedhere 
\end{align*} 
\end{proof}

Less straightforward is that when
applying the above result with 
$$
		  \sigma_0=\sigma_1=k^{-4}
		  \triangleright \cdot\,,\quad
		  \sigma_2=k^{-2}
		  \triangleright \cdot\,,\quad
		  \sigma_3=\mathrm{id},
$$
$$
		  \partial_1=(k^{-4}
		  \triangleright \cdot) \circ \partial_H
		   ,\quad
		  \partial_2=(k^{-3} \triangleright
		  \cdot) \circ \partial_e,\quad
		  \partial_3=(k^{-1} \triangleright
		  \cdot) \circ \partial_f 
$$
and a suitable twisted trace, one obtains a
cohomologically nontrivial
$\vartheta^{-1}$-twisted cocycle.

\begin{lemma}[{\cite[Corollary~3.8]{HK2}}]\label{haddock}
Define a linear functional 
$\int_{[1]} : \A \rightarrow \mathbb{C}$ by
$$
		  \int_{[1]} a^{n} b^{m} c^{r} := 
		  \delta_{n,0} \delta_{m,0} \delta_{r,0},\quad 
		  \int_{[1]} b^{m} c^{r} d^{s} := 
		  \delta_{m,0} \delta_{r,0} \delta_{s,0}. 
$$
Then $\int_{[1]}$ is a $\sigma_{L}^{2} \circ \vartheta^{-1}$-twisted trace, and
the cochain $\varphi \in C^3$ given by
\enveqn{
\varphi(a_{0}, \ldots, a_{3}) = 
\int_{[1]} \left(k^{-4} \triangleright (a_{0} \,\partial_H( a_{1})) \right) \left(k^{-3} \triangleright \pe(a_{2}) \right) \left(k^{-1} \triangleright \pf(a_{3}) \right)
}
is a cocycle,
 $b_{\vartheta^{-1}} \varphi = 0$, whose
cohomology class is nontrivial, 
$b_{\vartheta^{-1}} \psi \neq \varphi$
 for all $\psi \in C^2$.
\end{lemma}

Later, we will also have to consider the
cocycles that are obtained by using 
the (twisted) derivations
$\partial_H,\partial_e,\partial_f$ in a different
order. Explicitly,
this is handled by the following result.

\begin{lemma}\label{salmon}
In the situation of Lemma~\ref{tuna},
 define 
$$
		  \tilde \partial_3 =
		  \sigma_1 \circ \sigma_2^{-1}
 \circ \partial_3,\quad
		  \tilde \partial_2:=
		  \partial_2 \circ \sigma_2^{-1}
 \circ \sigma_3,\quad
		  \hat \partial_2:=\sigma_0 \circ
 \sigma_1^{-1} \circ \partial_2,\quad
		  \hat \partial_1:=\partial_1
 \circ \sigma_1^{-1} \circ \sigma_2.
$$
Then we have
\begin{align*}
		  \int \smallsmile \partial_1
 \smallsmile \partial_2 \smallsmile
 \partial_3 +
		  \int \smallsmile \partial_1
 \smallsmile \tilde \partial_3 \smallsmile
 \tilde \partial_2 
&= b_{\vartheta^{-1}} \psi_{132},\\
		  \int \smallsmile \partial_1
 \smallsmile \partial_2 \smallsmile
 \partial_3 +
		  \int \smallsmile \hat\partial_2
 \smallsmile \hat \partial_1 \smallsmile
 \partial_3 
&= b_{\vartheta^{-1}} \psi_{213},
\end{align*} 
where 
\begin{align*}
		  \psi_{132} (a_0,a_1,a_2)
&:=
		  \int \sigma_0(a_0)
		  \partial_1(a_1)
		  \partial_2
 (\sigma_2^{-1}(\partial_3(a_2))),\\
		  		 \psi_{213} (a_0,a_1,a_2)
&:=
		  -\int \sigma_0(a_0)
		  \partial_1 (\sigma_1^{-1}(\partial_2(a_1)))\partial_3(a_2)
		  .  
\end{align*} 
\end{lemma}
\begin{proof}
Straightforward computation.
\end{proof}

Applying Lemma \ref{salmon} repeatedly to the cocycle
$\varphi$ from Lemma~\ref{haddock} gives cohomologous cocycles.

\begin{corl}\label{cor:perms}
The cocycle $\varphi$ from
Lemma~\ref{haddock} is cohomologous to each of
$$
		  \varphi_{132}
		  (a_0,a_1,a_2,a_3):=
		  -q^{-2} 
		  \int_{[1]}
		  \left(k^{-4} \triangleright (a_{0} \,
		  \partial_H ( a_{1})) \right)
		  \left(k^{-3} \triangleright
		  \pf(a_{2}) \right) 
		  \left(k^{-1} \triangleright \pe(a_{3}) \right),
$$
$$
		  \varphi_{213} (a_0,a_1,a_2,a_3):=
		  -\int_{[1]} 
		  \left(k^{-4} \triangleright
		  a_{0}\right)
		  \left(k^{-3} \triangleright
		  \partial_e(a_1)\right)
		  \left(k^{-2} \triangleright \partial_H (a_2)\right)
		  \left(k^{-1} \triangleright
		  \partial_f(a_3)\right),		  
$$
$$
		  \varphi_{312}
		  (a_0,a_1,a_2,a_3):=
		  q^{-2} 
		  \int_{[1]}
		  \left(k^{-4} \triangleright
		  a_{0}\right)
		  \left(k^{-3} \triangleright
		  \partial_f(a_1)\right)
		  \left(k^{-2} \triangleright
		  \partial_H ( a_2)\right)
		  \left(k^{-1} \triangleright
		  \partial_e(a_3)\right),		  
$$
$$
		 \varphi_{231} (a_0,a_1,a_2,a_3):=
		  \int_{[1]} 
		  \left(k^{-4} \triangleright
		  a_{0}\right)
		  \left(k^{-3} \triangleright
		  \partial_e(a_1)\right)
		  \left(k^{-1} \triangleright
		  \partial_f(a_2)\right)
		  \left(\partial_H  (a_3)\right)		  
$$
and 
$$
		  \varphi_{321}
		  (a_0,a_1,a_2,a_3):=
		  -q^{-2} 
		  \int_{[1]}
		  \left(k^{-4} \triangleright
		  a_{0}\right)
		  \left(k^{-3} \triangleright
		  \partial_f(a_1)\right)
		  \left(k^{-1} \triangleright
		  \partial_e(a_2)\right)\left(
		  \partial_H ( a_3)\right).		  
$$
\end{corl}

\begin{proof}
To begin, one applies
 Lemma~\ref{salmon} to $\varphi$ with
$$
		  \tilde \partial_3 =
		  (k^{-3}
		  \triangleright \cdot) \circ \partial_f,\quad
		  \tilde \partial_2=
		  (k^{-3} \triangleright
		  \cdot) \circ \partial_e \circ (k^{2}
		  \triangleright \cdot),\quad
		  \hat \partial_2=(k^{-3} \triangleright
		  \cdot) \circ \partial_e,\quad
		  \hat \partial_1:=(k^{-4} \triangleright
		  \cdot) \circ \partial_H
		  (\cdot) 
 \circ (k^{2} \triangleright
		  \cdot)\,.
$$
The formulae for these derivations can be simplified by commuting $\partial_e$ and 
$k\,\triangleright$ to obtain
$$
		  \tilde \partial_3 =
		  (k^{-3}
		  \triangleright \cdot) \circ \partial_f,\quad
		  \tilde \partial_2=
		  q^{-2} (k^{-1} \triangleright
		  \cdot) \circ \partial_e,\quad
		  \hat \partial_2=(k^{-3} \triangleright
		  \cdot) \circ \partial_e,\quad
		  \hat \partial_1:=(k^{-2} \triangleright
		  \cdot) \circ \partial_H(
		  \cdot).
$$
This gives $\varphi_{132}$ and $\varphi_{213}$.
Then we can apply Lemma~\ref{salmon}
again to  $\varphi_{213}$.
Going from $\varphi_{213}$ to $\varphi_{312}$ is easy, since it only involves
exchanging $e$ and $f$.
Next we obtain $\varphi_{231}$ from $\varphi_{213}$ by applying
 Lemma~\ref{salmon} 
with
$$
		  \sigma_0=k^{-4}
		  \triangleright \cdot,\quad\sigma_1=
		  \sigma_2=k^{-2}
		  \triangleright \cdot,\quad
		  \sigma_3=\mathrm{id},
$$
$$
		  \partial_1=(k^{-3} \triangleright
		  \cdot) \circ \partial_e,\quad
		  \partial_2=(k^{-2}
		  \triangleright \cdot) \circ \partial_H
		  ,\quad
		  \partial_3=(k^{-1} \triangleright
		  \cdot) \circ \partial_f 
$$
which gives
$$
		  \tilde \partial_3=\sigma_1 \circ
 \sigma_2^{-1} \circ \partial_3=\partial_3=
		  (k^{-1} \triangleright
		  \cdot) \circ \partial_f,
$$
$$
		  \tilde \partial_2=\partial_2
 \circ \sigma_2^{-1} \circ \sigma_3=
		  (k^{-2}
		  \triangleright \cdot) \circ \partial_H
		   \circ (k^{2}
		  \triangleright \cdot)=\partial_H\,.
$$

The last cocycle is obtained analogously  from $\varphi_{312}$.
\end{proof}



A homologically nontrivial 3-cycle $dvol$ in the (pre)dual chain complex
$C_\bullet:=\A^{\otimes_\mathbb{C}\bullet+1}$ (with differential dual to $b_{\vartheta^{-1}}$)
has been computed in \cite{HK1,HK2}:

\enval{
dvol &:= d \ox a \ox b \ox c - d \ox a \ox c \ox b + q \, d \ox c \ox a \ox b \nonumber \\
& \quad - q^{2} \, d \ox c \ox b \ox a + q^{2} \, d \ox b \ox c \ox a - q \, d \ox b \ox a \ox c \nonumber \\
& \quad + c \ox b \ox a \ox d - c \ox b \ox d \ox a + q \, c \ox d \ox b \ox a \nonumber \\
& \quad - c \ox d \ox a \ox b + c \ox a \ox d \ox b - q^{-1} \, c \ox a \ox b \ox d \nonumber \\
& \quad + (q^{-1}-q) \, c \ox b \ox c \ox b \label{eqn_dvol}
}

With this normalisation, we have
$\varphi(dvol)=1$.

\section{Some meromorphic functions}
\label{sec:mero}

In this section 
we demonstrate that certain functions have meromorphic
continuations. These functions arise in the residue formula for the Hochschild cocycle in
the next two sections. We require the following notation.
For any $l \in \frac{1}{2} \mathbb{N}_0$ 
and $-(2l+1) \leq n \leq (2l+1)$ define

\enveq{
\lambda_{l,n} := 
\sqrt{\left(\frac{n}{2}\right)^{2} + q^{n}\left(\left[ l+\half \right]_{q}^{2} - \left[\frac{n}{2} \right]_{q}^{2} \right)}\,.
}

We also define the finite sets

\enveqn{
\mathcal{J}_{l} := \begin{cases}
\{ 0, 2, \ldots, 2l-1 \} & \quad l \in (\bN_{0} + \thalf) \\
\{ 1, 3, \ldots, 2l-1 \} & \quad l \in \bN
\end{cases}.
}

\begin{lemma}
\label{lemma_holomorphic}

The formulas

\enveqn{
z \mapsto f_1(z):=\sum_{2l = 1}^{\infty} \sum_{i = -l}^{l} \sum_{n \in \mathcal{J}_{l}} 
\frac{q^{2l-2i}}{ (1 + \lambda_{l, n}^{2})^{z/2} }
}

\enveqn{
z \mapsto f_2(z):=\sum_{2l = 1}^{\infty} \sum_{i = -l}^{l} \sum_{n \in \mathcal{J}_{l}} 
\frac{q^{2l-n}}{ (1 + \lambda_{l, n}^{2})^{z/2} }
}

define holomorphic functions on $\domain_2$, where we abbreviate
$$
\domain_t := \{z \in \mathbb{C} \mid 
\mathrm{Re}(z) > t\},\quad t \in \mathbb{R}.
$$
\end{lemma}

\begin{proof}
We will show that the sums converge uniformly on compacta. 
To begin with, we take
$z=t \in (2,\infty)$, and compute the summation over the $i$ parameter for $f_{1}$ and $f_{2}$ giving

\enval{
f_{1}(t) 
&= \sum_{2l = 1}^{\infty} \sum_{n \in \mathcal{J}_{l}} 
\frac{q^{2l} [2l+1]_{q} }{ (1 + \lambda_{l, n}^{2})^{t/2} }, &
f_{2}(t) 
&= \sum_{2l = 1}^{\infty} \sum_{n \in \mathcal{J}_{l}} \frac{(2l+1) q^{2l-n}}{ (1 + \lambda_{l, n}^{2})^{t/2} }.
\label{eq:pacific-giant-octopus}
}

For $l \in \frac{1}{2}\mathbb{N}_0$ and $n\in \J_l$ we have the inequality
$$
		  \left[l+\frac{1}{2}
 \right]^2_q-\left[\frac{n}{2}
 \right]^2_q \ge 
		  [2l]_q
$$
with equality attained for $n=2l-1$. This inequality implies

\begin{equation}
1 + \lambda_{l, n}^{2} \geq 1 + \left(\frac{n}{2}\right)^{2} + q^{n}[2l]_{q} 
\geq 1 + \left(\frac{n}{2}\right)^{2} + q^{n-2l+1}.
\label{eq:mullet}
\end{equation}

Since the summands in Equation \eqref{eq:pacific-giant-octopus} are positive, we
may invoke Tonelli's theorem to rearrange the order of summation
$$
\sum_{2l = 1}^{\infty} \sum_{n \in \mathcal{J}_{l}} 
\rightarrow \sum_{n = 0}^{\infty} \sum_{l = (n+1)/2}^{\infty}.
$$ 
Combining the elementary inequality $q^{2l} [2l+1]_{q} \leq q^{-1}Q$ with  
Equation \eqref{eq:mullet} gives the inequalities

\envaln{
f_{1}(t) &\leq q^{-1} Q \sum_{n = 0}^{\infty} 
\sum_{l = \frac{n+1}{2}}^{\infty} \frac{1 }{ (1 + \left(\frac{n}{2}\right)^{2} + q^{n-2l+1})^{t/2} }, &
f_{2}(t) &\leq \sum_{n = 0}^{\infty} 
\sum_{l = \frac{n+1}{2}}^{\infty} \frac{(2l+1) q^{2l-n}}{ (1 + \left(\frac{n}{2}\right)^{2} + q^{n-2l+1})^{t/2} }.
}

We reparameterise the sums defining $f_1$ and $f_2$ using $y = 2l-1-n$ with 
summation range $y=0$ to $y=\infty$. This yields

\enval{
f_{1}(t) &\leq q^{-1} Q \sum_{n = 0}^{\infty} \sum_{y = 0}^{\infty} \frac{1 }{ (1 + \left(\frac{n}{2}\right)^{2} + q^{-y})^{t/2} }, &
f_{2}(t) &\leq \sum_{n = 0}^{\infty} \sum_{y = 0}^{\infty} \frac{(y+n+2) q^{y+1}}{ (1 + \left(\frac{n}{2}\right)^{2} + q^{-y})^{t/2} }.
\label{eq:barramundi}
}

Next we employ the inequality $\alpha^{2} + \beta^{2} \geq \alpha \beta$, valid for 
any positive real numbers $\alpha$ and $\beta$, to $f_{1}(t)$. This yields

\enveqn{
f_{1}(t) \leq q^{-1} Q \sum_{n =
 0}^{\infty} \sum_{y = 0}^{\infty} \,q^{yt/4}\,
\left(1 + \left(\frac{n}{2}\right)^{2}\right)^{-t/4}  < \infty \qquad \text{for all} \ t > 2.
}

For the function $f_{2}(t)$, we evaluate the sums over $y$ on the right hand side to obtain, for some 
positive constants $C_{1}$ and $C_{2}$,

\enveqn{
f_{2}(t) \leq \sum_{n = 0}^{\infty}
 \sum_{y = 0}^{\infty} \frac{(y+n+2)
 q^{y+1}}
{ \left(1 +  \left(\frac{n}{2}\right)^{2}\right)^{t/2} } 
= \sum_{n = 0}^{\infty} \frac{C_{1} + C_{2}n }{ \left(1 +  \left(\frac{n}{2}\right)^{2}\right)^{t/2} }.
}

This last sum is finite for all $t > 2$, and bounded uniformly for $t\geq 2+\epsilon$ for any $\epsilon>0$. 
This establishes that
$f_1,f_2$ are finite for all $\mathrm{Re}(z) > 2$, and the sums defining
them converge uniformly on vertical strips, and so on compacta. 
Finally, to show that $f_1,f_2$
are holomorphic in the half-plane ${\mathrm Re}(z)>2$, we invoke the
Weierstrass convergence theorem.
\end{proof}

\begin{lemma}
\label{lem:general-pole-3}
For any positive reals $x,\, y,\, r > 0$, $w\in \N$, and $z \in
\domain_3$, define 

\enveqn{
  h(z) := 
  \sum_{n=1}^{\infty}\sum^\infty_{m=w} 
  \frac{e^{rm} }{ (x^{2} n^{2} + y^{2} e^{rm})^{z/2} }
}

Then we have:
\begin{enumerate}
\item $h$ is a holomorphic function on $\domain_3$; 
\item $h$ has a meromorphic continuation to 
$\domain_2$ with a simple pole
at $z=3$;
\item This continuation can be written as 
\enveqn{
h(z) = \frac{\sqrt{\pi}}{2xy^{z-1}} \frac{\Gamma(\frac{z-1}{2})}{\Gamma(\frac{z}{2})} 
\frac{e^{-rw(z-3)/2}}{1 - e^{-r(z-3)/2}} - \frac{1}{2y^{z}} \frac{e^{-rw(z-2)/2}}{1 - e^{-r(z-2)/2}} + err(z)
}
where $err$ is a holomorphic function on $\domain_2$
that satisfies 
$$
|err(z)| \leq \frac{1}{2y^{\mathrm{Re}(z)}} \frac{e^{-rw(\mathrm{Re}(z) - 2)/2}}{1 - e^{-r(\mathrm{Re}(z) - 2)/2}}.
$$
\end{enumerate} 
\end{lemma}

\begin{proof} 
Until further notice, we take $z$ real and positive. Later we will extend our
results to complex $z$ as in Lemma \ref{lemma_holomorphic}.
Inserting the Mellin transform of $f(t)=e^{-(x^2n^2+y^2e^{rm})t}$ gives


\enveqn{
h(z) = \sum_{n=1}^\infty \sum_{m=w}^{\infty} 
\frac{e^{rm} }{\Gamma(\frac{z}{2})} \int_{0}^{\infty} t^{\frac{z}{2} - 1}e^{-tx^{2} n^{2}} e^{-t y^{2} e^{rm}} dt.
}

For $z$ real, all terms above are positive. Therefore we can apply Tonelli's theorem to 
exchange the order of integration with summation. 
Having done this, we  consider the sum $\sum_{n = 1}^{\infty} e^{-tx^{2} n^{2}}$. The Poisson summation 
formula provides the identity

\enveqn{
\sum_{n = 1}^{\infty} e^{-tx^{2} n^{2}} = \frac{1}{2} \left( \sqrt{\frac{\pi}{tx^{2}}} 
\left( 1 + 2 \sum_{n = 1}^{\infty} e^{-\frac{n^{2}\pi^{2}}{tx^{2}}} \right) - 1 \right).
}

Substituting this identity into the expression for $h(z)$ we find

\envaln{
h(z) &= \frac{1}{2} \sum_{m=w}^{\infty} \frac{e^{rm}}{(y^{2} e^{rm})^{\frac{z}{2}}} \left( \frac{\sqrt{\pi}}{x} 
\frac{\Gamma(\frac{z-1}{2})}{\Gamma(\frac{z}{2})} (y^{2} e^{rm})^{\frac{1}{2}}  - 1 \right) \\
& \quad + \frac{\sqrt{\pi}}{x} \sum_{n=1}^\infty \sum_{m=w}^{\infty} \frac{e^{rm} }{\Gamma(\frac{z}{2})} 
\int_{0}^{\infty} t^{\frac{z-1}{2} - 1}e^{-\frac{n^{2}\pi^{2}}{tx^{2}}} e^{-t y^{2} e^{rm}} dt.
}

To explore the convergence of the double sum we denote

\enveqn{
g_{n}(s) := \int_{0}^{\infty} t^{\frac{z-1}{2} - 1}e^{-\frac{n^{2}\pi^{2}}{tx^{2}}} e^{-t s} dt.
}

Later we will set $s = y^{2} e^{rm} > 0$, so we consider only positive, 
real $s$, making $g_{n}(s)$ a positive real function. 
Using 
\cite[Section 26:14]{OS}  to evaluate this Laplace transform gives

\enveqn{
g_{n}(s) = 
2 \left( \frac{n \pi}{x \sqrt{s}} \right)^{\frac{z-1}{2}} K_{\frac{z-1}{2}} \left( \frac{2 n \pi \sqrt{s}}{x} \right)
}

where $u\mapsto K_{\nu}(u)$ is the modified Bessel function of the second kind. For $u > 0$ and real $\nu > 1/2$,
$u^{\nu}K_{\nu}(u)$ is positive, as both $u^{\nu}$ and $K_{\nu}(u)$ are positive. Also, 
the derivative (referring again to \cite{OS}) is given by

\enveqn{
\frac{\p}{\p u}\left( u^{\nu}K_{\nu}(u) \right) = -u^{\nu} K_{\nu - 1}(u) \leq 0 \qquad \text{for all} \ u \geq 0.
}

Thus the function $u\mapsto u^{\nu}K_{\nu}(u)$ is positive and monotonically decreasing 
for all $u > 0$. Hence for all $\epsilon > 0$ we have the bound

\enveq{
\epsilon \sum_{n = 1}^{\infty} (\epsilon n)^{\nu} K_{\nu}(\epsilon n) \leq \int_{0}^{\infty} u^{\nu} K_{\nu}(u) du.
\label{eq:bound}
}

Evaluating the integral (using \cite[Chapter 51]{OS}) yields

\enveqn{
\sum_{n = 1}^{\infty} (\epsilon n)^{\nu} K_{\nu}(\epsilon n) \leq \frac{1}{\epsilon} 2^{\nu - 1} 
\Gamma(\tfrac{1}{2}) \Gamma(\nu + \tfrac{1}{2}).
}

If we now set $s = y^{2} e^{rm}$, we obtain the bound

\enveqn{
  \sum_{n=1}^\infty \sum_{m=w}^{\infty}\frac{e^{rm} }{\Gamma(\frac{z}{2})} 
\int_{0}^{\infty} t^{\frac{z-1}{2} - 1}e^{-\frac{n^{2}\pi^{2}}{tx^{2}}} e^{-t y^{2} e^{rm}} dt 
\leq 2\sum_{n=1}^\infty \sum_{m=w}^{\infty} \frac{e^{rm} }{\Gamma(\frac{z}{2})}  
\left( \frac{n \pi}{x ye^{rm/2}} \right)^{\frac{z-1}{2}} K_{\frac{z-1}{2}} \left( \frac{2 n \pi ye^{rm/2}}{x} \right).
}

Now estimating the sum over $n$ on the right using Equation \eqref{eq:bound} gives us

\envaln{
2\sum_{n=1}^{\infty}  \left( \frac{n \pi}{x \sqrt{s}} \right)^{\frac{z-1}{2}} K_{\frac{z-1}{2}} 
\left( \frac{2 n \pi \sqrt{s}}{x} \right) &= 2 \left( \frac{1}{2s} \right)^{\frac{z-1}{2}} 
\sum_{n=1}^{\infty} \left( \frac{2n \pi \sqrt{s}}{x} \right)^{\frac{z-1}{2}} K_{\frac{z-1}{2}} 
\left( \frac{2 n \pi \sqrt{s}}{x} \right) \\
& \leq 2 \left( \frac{1}{2s} \right)^{\frac{z-1}{2}} \frac{x}{2 \pi \sqrt{s}} 2^{\frac{z-1}{2} - 1} 
\Gamma(\tfrac{z}{2}) \Gamma(\tfrac{1}{2}) \\
&= \frac{x \Gamma(\tfrac{1}{2})\Gamma(\tfrac{z}{2})}{2 \pi }  \frac{1}{s^{z/2}}
=\frac{x \Gamma(\tfrac{1}{2})\Gamma(\tfrac{z}{2})}{2 \pi }  \frac{1}{y^ze^{zrm/2}}.
}

Hence by summing the remaining geometric series in $m$ we obtain the bound

\envaln{
 \sum_{n=1}^\infty \sum_{m=w}^{\infty} \frac{e^{rm} }{\Gamma(\frac{z}{2})} 
\int_{0}^{\infty} t^{\frac{z-1}{2} - 1}e^{-\frac{n^{2}\pi^{2}}{tx^{2}}} e^{-t y^{2} e^{rm}} dt 
&\leq \frac{\Gamma \left(\frac{z}{2} \right)}{\Gamma(\frac{z}{2})} 
\frac{x \Gamma(\tfrac{1}{2})}{y^{z} 2 \pi } \sum_{m=w}^{\infty} \frac{e^{rm}}{e^{rmz/2}} \\
& \leq \frac{x \Gamma(\tfrac{1}{2})}{y^{z} 2 \pi } 
\frac{e^{-rw(z - 2)/2}}{1 - e^{-r(z - 2)/2}}.
}

Evaluating the remaining geometric series in $h(z)$ as above, we arrive at

\enveq{
h(z) = \frac{\sqrt{\pi}}{2xy^{z-1}} \frac{\Gamma(\frac{z-1}{2})}{\Gamma(\frac{z}{2})} 
\frac{e^{-rw(z-3)/2}}{1 - e^{-r(z-3)/2}} - \frac{1}{2y^{z}} \frac{e^{-rw(z-2)/2}}{1 - e^{-r(z-2)/2}} + err(z)
\label{eq:box-jelly}
}

where

\envgan{
err(z) := \frac{\sqrt{\pi}}{x} \sum_{n=1}^\infty \sum_{m=w}^{\infty} \frac{e^{rm} }{\Gamma(\frac{z}{2})} 
\int_{0}^{\infty} t^{\frac{z-1}{2} - 1}e^{-\frac{n^{2}\pi^{2}}{tx^{2}}} e^{-t y^{2} e^{rm}} dt, \\
err(z) \leq \frac{1}{2y^{z}} \frac{e^{-rw(z - 2)/2}}{1 - e^{-r(z - 2)/2}}.
}

Thus the sum defining the function $err$ converges for all $z > 2$, and this
convergence is uniform on compact intervals. Now we observe that for $z\in\C$ we 
have $|h(z)|\leq h(|z|)$ and
similarly $|err(z)|\leq err(|z|)$. Hence the sums defining $h$ converge uniformly on closed vertical 
strips in the half-plane $\domain_3$, and so on compacta. Similarly the sums and integral defining
$err$ converge uniformly on compact subsets of the half-plane $\domain_2$.

Hence the Weierstrass convergence theorem implies
that $err$ is holomorphic on the half-plane $\domain_2$ and that $h$ is holomorphic 
on $\domain_3$.
Moreover the formula for $h$, Equation \eqref{eq:box-jelly}, provides a meromorphic continuation
of $h$ to the half-plane $\domain_2$.
\end{proof}



\begin{lemma}
\label{lemma_tau_residue}
The formula
\enveqn{
f(z):=  \sum_{n=0}^{\infty} \sum_{l = \frac{n+1}{2}}^{\infty} \frac{q^{n-2l} }{ (1 + \lambda_{l, n}^{2})^{z/2} }
}

defines a holomorphic function on $\domain_3$. Moreover $f$ has a meromorphic continuation 
to $\domain_2$, a simple pole at $z = 3$ with
residue $4qQ^{-2}/\ln(q^{-1})$.
\end{lemma}

\begin{proof}
First we write

\enveqn{
1 + \lambda_{l, n}^{2} = 1 +\tfrac{n^{2}}{4} + q^{n} \left( \left[ l + \tfrac{1}{2} \right]^{2} - \left[ \tfrac{n}{2} \right]^{2} \right) = \tfrac{1}{4} n^{2} + Q^{2} q^{-1} q^{n-2l} + C_{n, l}
}

where $C_{n,l}$ is uniformly bounded in $n,l$, and is given by

\enveqn{
C_{n, l} = 1 + Q^{2} q^{n} (q^{2l+1} - 2) - q^{n} \left[ \tfrac{n}{2} \right]^{2},\qquad |C_{n,l}|\leq 1+3Q^2.
}

Now we reparametrise the summation by letting $m = 2l-n$, yielding

\enveqn{
f(z) = 
\sum_{n=0}^{\infty} \sum_{m = 1}^{\infty} 
\frac{q^{-m} }{ (\tfrac{1}{4} n^{2} + Q^{2} q^{-1} q^{-m} + C_{n, m})^{z/2} }
}

where we understand $C_{n, m} = C_{n, l=(n+m)/2}$. The function

\enveqn{
z \mapsto \sum_{m = 1}^{\infty} \frac{q^{-m} }{ (Q^{2} q^{-1} q^{-m} + C_{0, m})^{z/2} } = \sum_{m = 1}^{\infty} \frac{q^{m(\frac{z}{2} - 1)} }{ (Q^{2} q^{-1}  + q^{m} C_{0, m})^{z/2} }
}

has summands with absolute value bounded by $M q^{m(\frac{z}{2} - 1)}$, $M>0$ constant, and so
by the Weierstrass convergence theorem is holomorphic for $\mathrm{Re}(z) > 2$. Hence for
some holomorphic function $holo$ on $\domain_2$ we have

\enval{
f(z) &= \sum_{n, m=1}^{\infty} \frac{q^{-m} }{ (\tfrac{1}{4} n^{2} + Q^{2} q^{-1} q^{-m} + C_{n, m})^{z/2} } 
+ holo(z) \nonumber \\
&= \sum_{n, m=1}^{\infty} \frac{q^{-m} }{ (\tfrac{1}{4} n^{2} + Q^{2} q^{-1} q^{-m})^{z/2}} \left( 1 + \frac{C_{n, m}}{\tfrac{1}{4} n^{2} + Q^{2} q^{-1} q^{-m}} \right)^{-z/2} + holo(z). \label{eqn_fz_expand}
}

The strategy now is to perform a binomial expansion on 
$$
\left( 1 + \frac{C_{n, m}}{\tfrac{1}{4} n^{2} + Q^{2} q^{-1} q^{-m}} \right)^{-z/2}
$$
ending up with a new sum of functions $\sum_{n,m,k}D_{n,m,k} \,h(z+2k)$ where $h$ is 
as in Lemma \ref{lem:general-pole-3}.
The binomial expansion requires the inequality

\enveqn{
\frac{C_{n, m}}{\tfrac{1}{4} n^{2} + Q^{2} q^{-1} q^{-m}} < 1
}

which holds for sufficiently large $m$. 
Recall that $|C_{n,m}|\leq 1+3Q^2=:C$ uniformly in 
$n,\,m$, and so we may choose $p \in \bN$ such that

\enveqn{
q^{-p} > q Q^{-2}C \qquad \implies \qquad \frac{|C_{n, m}|}{\tfrac{1}{4} n^{2} + Q^{2} q^{-1} q^{-m}} < 1 
\quad \forall n \geq 1,\  m \geq p.
}

Now, for any fixed $p$, sums of the form

\enveqn{
\sum_{n=1}^{\infty} \sum_{m=1}^{p-1} 
\frac{q^{-m} }{ (\tfrac{1}{4} n^{2} + Q^{2} q^{-1} q^{-m} + C_{n, m})^{z/2}}
}

can immediately be seen to be holomorphic for $\mathrm{Re}(z) > 2$ 
as the sum can be bounded by a constant multiple of the Riemann zeta function. 
Hence for such a choice of $p\in \N$ and for some holomorphic function
$holo$ on $\domain_2$ we have

\envaln{
f(z) &= \sum_{n=1}^{\infty} \sum_{m=p}^{\infty} \frac{q^{-m} }{ (\tfrac{1}{4} n^{2} + Q^{2} q^{-1} q^{-m})^{z/2}} \left( 1 + \frac{C_{n, m}}{\tfrac{1}{4} n^{2} + Q^{2} q^{-1} q^{-m}} \right)^{-z/2} + holo(z).
}
Now we perform the binomial expansion, separating the 
resulting infinite sum $\sum_{k=0}^\infty$ into the 
$k=0$ term and $\sum_{k=1}^\infty$. This gives

\envaln{
f(z)&= \sum_{n=1}^{\infty} \sum_{m=p}^{\infty} \frac{q^{-m} }{ (\tfrac{1}{4} n^{2} + Q^{2} q^{-1} q^{-m})^{z/2}} + \sum_{k = 1}^{\infty} \bca -\frac{z}{2} \\ k \eca \sum_{n=1}^{\infty} \sum_{m=p}^{\infty} \frac{q^{-m} (C_{n, m})^{k} }{ (\tfrac{1}{4} n^{2} + Q^{2} q^{-1} q^{-m})^{\frac{z+2k}{2}}} + holo(z)\\
&= h(z) + \sum_{k = 1}^{\infty} \bca -\frac{z}{2} \\ k \eca \sum_{n=1}^{\infty} 
\sum_{m=p}^{\infty} \frac{q^{-m} (C_{n, m})^{k} }{ (\tfrac{1}{4} n^{2} + Q^{2} q^{-1} q^{-m})^{\frac{z+2k}{2}}} 
+ holo(z),
}

where $h$ is as in Lemma \ref{lem:general-pole-3}, with $x=1/2$, $y=q^{-1/2}Q$,
$r=\ln(q^{-1})$ and $w=p$. Our aim now is to show that $f-h$ is a
holomorphic function on $\domain_2$. We need to show that the remaining summation converges 
to such a function. This remaining sum is bounded by
\begin{align*}
&\left|\sum_{k = 1}^{\infty} \bca -\frac{z}{2} \\ k \eca \sum_{n=1}^{\infty} 
\sum_{m=p}^{\infty} \frac{q^{-m} (C_{n, m})^{k} }{ (\tfrac{1}{4} n^{2} + Q^{2} q^{-1} q^{-m})^{\frac{z+2k}{2}}}\right|\\
&\qquad\leq \sum_{k = 1}^{\infty} \left| \bca -\frac{z}{2} \\ k \eca\right| \,C^k
\sum_{n=1}^{\infty} 
\sum_{m=p}^{\infty} \frac{q^{-m}}{ (\tfrac{1}{4} n^{2} + Q^{2} q^{-1} q^{-m})^{\frac{{\mathrm Re}(z)+2k}{2}}}\\
&\qquad\qquad= \sum_{k = 1}^{\infty} \left| \bca -\frac{z}{2} \\ k \eca\right| \,C^k h({\mathrm Re}(z)+2k).
\end{align*}

To estimate this sum of functions, we infer from Lemma \ref{lem:general-pole-3} that there exists
a positive function $M$ which is defined for $\mathrm{Re}(z) > 3$ and such that

\enveqn{
|h(z)| \leq M(z) \frac{e^{-\mathrm{Re}(z)rp/2}}{y^{\mathrm{Re}(z)}} 
=  M(z) (q^{\frac{1}{2}(p+1)}Q^{-1})^{\mathrm{Re}(z)}.
}

Hence

\enveqn{
\left| \sum_{k = 1}^{\infty} \bca -\frac{z}{2} \\ k \eca \sum_{n=1}^{\infty} 
\sum_{m=p}^{\infty} \frac{q^{-m} (C_{n, m})^{k} }{ (\tfrac{1}{4} n^{2} + Q^{2} q^{-1} q^{-m})^{\frac{z+2k}{2}}} \right| \leq \sum_{k = 1}^{\infty} \left| \bca -\frac{z}{2} \\ k \eca \right| 
C^{k} M(z+2k) (q^{\frac{1}{2}(p+1)}Q^{-1})^{\mathrm{Re}(z)+2k}.
}

Recall that $p$ was chosen such that $q^{-p} > q Q^{-2}C$. Also 
the function $z\mapsto M(z)$ is uniformly bounded for ${\mathrm Re}(z)\geq 4$.
Hence, for all $z$ with ${\mathrm Re}(z)\geq 2$, the function $k\mapsto M(z+2k)$ is uniformly
bounded in $k$, by ${\bf M}$ say.
It thus follows that the sum

\begin{align*}
\sum_{k = 1}^{\infty} \left| \bca -\frac{z}{2} \\ k \eca \right| C^k
M(z+2k)(q^{\frac{1}{2}(p+1)}Q^{-1})^{\mathrm{Re}(z)+2k}
&\leq {\bf M} \sum_{k = 1}^{\infty} \left| \bca -\frac{z}{2} \\ k \eca \right|(q^{p+1}Q^{-2}C)^{k}
\end{align*}

converges for $\mathrm{Re}(z) > 2$, by comparing with the binomial expansion on the right hand side. 
The convergence is again uniform on compacta, so invoking
Weierstrass' convergence theorem
we conclude that $f(z) - h(z)$ is holomorphic for $\mathrm{Re}(z) > 2$. Hence there exists a function 
$holo$ which is defined and holomorphic for ${\mathrm Re}(z)>2$ such that 

\enveqn{
f(z) = \frac{\sqrt{\pi}}{(q^{-\frac{1}{2}}Q)^{z-1}} 
\frac{\Gamma(\frac{z-1}{2})}{\Gamma(\frac{z}{2})} \frac{q^{p(z-3)/2}}{1 - q^{(z-3)/2}} + holo(z)
}

So we see $f(z)$ is holomorphic for $\mathrm{Re}(z) > 3$, meromorphic for $\mathrm{Re}(z)>2$ and has a a simple pole 
at $z = 3$ with residue $4qQ^{-2}/\ln(q^{-1})$.
\end{proof}



\section{An analogue of a spectral triple}
\label{sec:spec}

We now introduce an analogue of a
spectral triple over $\A$. 
Let $\cH := \cH_{h} \oplus
		\cH_{h}$ be the Hilbert space given by two copies of the 
GNS space $\H_h=L^{2}(A, h)$. We define
a grading on $\H$ by
$\Gamma = \bma 1 & 0 \\ 0 & -1
\ema$. For any operator $\omega$ on $\H$
we abbreviate
\begin{equation}
		  \omega^+:=\frac{1+\Gamma}{2}
		  \omega  \frac{1+\Gamma}{2},\quad
		  \omega^-:=\frac{1-\Gamma}{2}
		  \omega  \frac{1-\Gamma}{2}.
\label{eq:giant-squid}
\end{equation}

The algebra $\A$ is represented on $\cH$ by

\enveqn{
\alpha \mapsto \bma \pi_{h}(\alpha) & 0 \\ 0 & \pi_{h}(\alpha) \ema
}

for $\alpha \in \A$. Here $\pi_{h}$ denotes the GNS representation by left 
multiplication on each copy of the
		space. In the sequel we will omit
		the symbol $\pi_h$.
We now introduce some  unbounded operators and
projections
 
\envaln{
\hat \Delta_{R} &= \bma \Delta_{R} & 0 \\ 0 &
\Delta_{R} \ema & \hat\Delta_{L} &= \bma
q^{-1} \Delta_{L} & 0 \\ 0 & q
\Delta_{L} \ema &
\Psi_{n} = \bma \Phi_{n+1} & 0 \\ 0 & \Phi_{n-1} \ema
}

on $\A \oplus
		\A \subset \cH$ and use
		them to define (on the same domain)

\enveqn{
\cD = \half \sum_{n = -\infty}^{\infty} \Psi_{n} \bma n & 0 \\ 0 & -n \ema + \hat\Delta_{L}^{\half} \bma 0 & \pe \\ \pf & 0 \ema.
}

We will see in the following lemma that
the commutators $[\cD, \alpha]$ of $\cD$ with algebra elements 
are not necessarily bounded, yet
unbounded in a very controlled manner. 
Even though $(\A,\H,\D)$ thus fails to
be a spectral triple, we will still be
able to construct 
an analytic expression for a residue Hochschild cocycle from the commutators.

\begin{lemma}
\label{lemma_triple_properties}
The triple $(\A, \cH, \cD)$ has the following properties:

\begin{enumerate}
\item The unbounded operator $\cD$ is
		essentially self-adjoint.
\item The commutator $[\cD, \alpha]$ is
		given by $\coms(\alpha) + \comt(\alpha) \hat\Delta_{L}$, 
where the linear maps $\coms, \comt \colon
		\A \rightarrow B(\cH)$
		are given by

\envaln{
\coms(\alpha) &= \partial_{H}(\alpha) \Gamma &
\comt(\alpha) &= \bma 0 & q^{-\half}
		\pe(\sigma_{L}^{-\half}(\alpha))
		\\ q^{\half}
		\pf(\sigma_{L}^{-\half}(\alpha)) &
		0 \ema. 
}
\end{enumerate}
\end{lemma}

\begin{proof}

First we recall from Section \ref{sec:mero} the numbers

\enveq{
\lambda_{l,n} := 
\sqrt{\left(\frac{n}{2}\right)^{2} + q^{n}\left(\left[ l+\half \right]_{q}^{2} - \left[\frac{n}{2} \right]_{q}^{2} \right)},
}

where $l \in \frac{1}{2} \mathbb{N}_0$ 
and $-(2l+1) \leq n \leq (2l+1)$. Also recall $ I_{l} := \{ -l, -l+1, \ldots, l-1,l \}$.
Then the set

\enveqn{
\left\{ \begin{pmatrix} 0 \\ \re{l}{i}{l} \end{pmatrix} , 
\begin{pmatrix} \re{l}{i}{-l} \\ 0 \end{pmatrix}, \begin{pmatrix} \re{l}{i}{j} \\ \stackrel{}{C^{l}_{j, \pm}} \re{l}{i}{j-1} 
\end{pmatrix} \colon l \in \half \bN_{0}, \ i \in I_{l}, \ j \in I_{l} \backslash \{-l\} \right\},
}

\enveqn{
where \ \ C^{l}_{j, \pm} = \frac{\pm \lambda_{l, 2j-1} - (j-\thalf) }{ q^{j-\half} \sqrt{\left[ l+\thalf \right]_{q}^{2} - \left[j - \thalf \right]_{q}^{2}} }
}

is an orthogonal basis for $\cH$ comprised of eigenvectors of $\cD$. 
The corresponding eigenvalues are $-(l+\thalf), -(l+\thalf)$ and $\pm \lambda_{l, 2j-1}$ respectively. 
This spectral representation establishes that $\cD$ is essentially self-adjoint.

Next, the commutator of $\cD$ with a homogeneous algebra element $\alpha = \Phi_{p} (\alpha)$, 
for some $p \in \bZ$, is computed directly. It is sufficient to consider just this case, because $\A$ consists of 
finite linear combinations of homogeneous elements (the generators are homogeneous). For such an element $\alpha$ we have

\envaln{
[\cD, \alpha] &= \half \sum_{n =
 -\infty}^{\infty} \Psi_{n} \alpha \bma
 n & 0 \\ 0 & -n \ema +
 \hat\Delta_{L}^{\half} \bma 0 & \pe \\
 \pf & 0 \ema \alpha \\ 
& \quad - \half \sum_{n = -\infty}^{\infty} \alpha \Psi_{n} \bma n & 0 \\ 0 & -n \ema - \alpha \hat\Delta_{L}^{\half} \bma 0 & \pe \\ \pf & 0 \ema.
}

It follows from the definition of the projections $\Phi_{n}$, now regarded as a linear operator on $\H_h$, that 
$\alpha \Phi_{n} = \Phi_{n+p} \alpha$ for any $n \in \bZ$. Using this, together 
with the definition of the derivations $\partial_e$ and $\partial_f$ in Equation \ref{eqn_twisted_ef}, the commutator simplifies to

\envaln{
[\cD, \alpha] &= \half \alpha \sum_{n =
 -\infty}^{\infty} \Psi_{n} 
\left(\bma n+p & 0 \\ 0 & -n-p \ema - \bma n & 0 \\ 0 & -n \ema \right) \\
& \quad + \hat\Delta_{L}^{\half} \bma 0 & (\pe(\alpha) \Delta_{L}^{\half} + \sigma_{L}^{\half}(\alpha) \pe) \\ (\pf(\alpha) \Delta_{L}^{\half} + \sigma_{L}^{\half}(\alpha) \pf) & 0 \ema - \alpha \hat\Delta_{L}^{\half} \bma 0 & \pe \\ \pf & 0 \ema.
}

Since $\sigma_{L}^{\half}(\alpha) = \hat\Delta_{L}^{-\half} \alpha \hat\Delta_{L}^{\half}$ as operators on $\A\oplus\A\subset\H$,  the last expression for the commutator simplifies to

\enveqn{
\hat\Delta_{L}^{\half} \bma 0 &
 \sigma_{L}^{\half}(\alpha) \pe \\ 
\sigma_{L}^{\half}(\alpha) \pf & 0 \ema = \alpha \hat\Delta_{L}^{\half} \bma 0 & \pe \\ \pf & 0 \ema,
}

and hence

\envaln{
[\cD, \alpha] &= \frac{p}{2} \alpha \bma 1 & 0 \\ 0 & -1 \ema + \hat\Delta_{L}^{\half} \bma 0 & \pe(\alpha) \Delta_{L}^{\half} \\ \pf(\alpha) \Delta_{L}^{\half} & 0 \ema \\
&= \partial_{H}(\alpha) \Gamma + \bma 0 & q^{-\half} \pe(\sigma_{L}^{-\half}(\alpha)) \\ q^{\half} \pf(\sigma_{L}^{-\half}(\alpha)) & 0 \ema \hat\Delta_{L}. \qedhere
}
\end{proof}

\section{The residue Hochschild cocycle}
\label{sec:res-hochs}

The main step in
the definition of the residue Hochschild cocycle
is the construction of a functional that plays the role of an integral. In 
the situations considered in the literature thus far, \cite{C, BeF, GVF, CNNR, CPRS1, KW}, 
functionals of the form
$$
T\mapsto \tau(T(1+\D^2)^{-z/2})
$$
were used, where $z\in \C$ and $\tau$ is
a faithful normal semifinite trace, or at worst a weight, on a von Neumann algebra containing the  
algebra of interest. Often, the von Neumann algebra is just $\B(\H)$, and
the functional $\tau$ is the operator trace.

In this example, we need to apply our functional to 
products of commutators $[\D,\alpha]\sim \hat{\Delta}_L$ with
$\alpha\in\A$, so it has to be defined on an algebra of unbounded
operators. We will deal with this using a cutoff that is defined
by the projections

\enveqn{
L_k:=\tilde{L}_k\oplus\tilde{L}_k,\quad \tilde{L}_{k}(\re{l}{i}{j}) := \begin{cases}
\re{l}{i}{j} & \quad l \leq k \\
0 & \quad \mbox{otherwise}
\end{cases}
}

and

\envaln{
P_{1} &= \sum_{n = 0}^{\infty} \Psi_{n} &
P_{2} \bca \re{l}{i}{j} \\ 0 \eca = (1-\delta_{j,-l}) \bca \re{l}{i}{j} \\ 0 \eca \quad & \quad
P_{2} \bca 0 \\ \re{l}{i}{j} \eca = (1-\delta_{j,l}) \bca 0 \\ \re{l}{i}{j} \eca.
}

Observe $P_{2}$ is the 
projection onto  
$\left(\mathrm{ker}\bma 0 & \pe \\ \pf & 0
\ema \right)^\perp$, and that
the projections $L_{k}$ converge strongly to the identity in $\B(\cH)$.

For $s\in \R^+$ we now define   a  functional $\Upsilon_{s}$
on positive operators $\omega \in\B(\cH) $ 
in the following way:

\enveqn{
\Upsilon_{s} (\omega) := 
\sup_{k \in \N} \mathrm{Tr} 
\left( P_{1} P_{2} L_{k} (1 +
\cD^{2})^{-s/4} \hat\Delta_{F}^{-\half}
\omega 
\hat\Delta_{F}^{-\half} (1 +
\cD^{2})^{-s/4} P_{1} P_{2}
L_{k}\right),\quad
\hat \Delta  _F=\hat \Delta_R \hat
\Delta _L
}

where $\mathrm{Tr}$ is the operator
trace on $\B(\cH)$. 
This expression
continues to make sense for possibly
unbounded positive operators defined on
and preserving the subspace 
$\A \oplus \A \subset \H$.  

\begin{lemma} 
\label{lemma_operator_trace}
For each $s\in \R_+$ the functional
 $\Upsilon_s$ is  positive and normal on
 $\B(\H)_+$.
It is faithful and semifinite on $P_1P_2\B(\cH)_+P_1P_2$.
\end{lemma}

\begin{proof}
We will compute the operator trace  using the Peter-Weyl basis 
$\left\{ \bca \re{l}{i}{j} \\ 0 \eca, \bca 0 \\ \re{l}{i}{j} \eca \right\}$ for $\cH$. 
The operators $(1 + \cD^{2})$, $\hat\Delta_{F}$,  $P_{1}$, $P_{2}$ and $L_{k}$ are 
all positive and diagonal in this basis. By using the definition of the operator trace, 
the value of the operators $\hat\Delta_{F}^{-1}$ and $(1 + \cD^{2})^{-s/4}$ on this basis, 
and the symmetry property for self-adjoint operators, we compute $\Upsilon_s(\omega)$ for 
$\omega\in \B(\H)_+$ (or even $\omega\geq 0$ and affiliated to $\B(\H)$) by

\envaln{
& \mathrm{Tr} \left(  P_{1} P_{2} L_{k}(1 + \cD^{2})^{-s/4} \hat\Delta_{F}^{-\half} \omega \hat\Delta_{F}^{-\half} (1 + \cD^{2})^{-s/4} P_{1} P_{2} L_{k}\right) = \\
&= \sum_{2l = 0}^{\infty}
 \sum_{i=-l}^{l} \sum_{j = -l}^{l}
 \frac{q^{-2i-(2j-1)}}{(1 + \lambda_{l,
 2j-1}^{2})^{s/2}} \frac{\la P_{1}^{+} P_{2}^{+} L_{k} \re{l}{i}{j}, \omega^+P_{1}^{+} P_{2}^{+} L_{k} \re{l}{i}{j} \ra}{\la \re{l}{i}{j}, \re{l}{i}{j} \ra} \\
& \quad + \sum_{2l = 0}^{\infty}
 \sum_{i=-l}^{l} \sum_{j = -l}^{l}
 \frac{q^{-2i-(2j+1)}}{(1 + \lambda_{l,
 2j+1}^{2})^{s/2}} \frac{\la 
P_{1}^{-} P_{2}^{-} L_{k}\re{l}{i}{j}, \omega^- P_{1}^{-} P_{2}^{-} L_{k} \re{l}{i}{j} \ra}{\la \re{l}{i}{j}, \re{l}{i}{j} \ra},
}
where $\omega^+$ and $\omega^-$ are as in Equation \eqref{eq:giant-squid}.
Now,

\envaln{
P_{1}^{+} P_{2}^{+} L_{k} \re{l}{i}{j} &= \begin{cases}
\re{l}{i}{j} & \quad \frac{1}{2} \leq j \leq l, \ \frac{1}{2} \leq l \leq k \\
0 & \quad \mbox{otherwise}
\end{cases} \\
P_{1}^{-} P_{2}^{-} L_{k} \re{l}{i}{j} &= \begin{cases}
\re{l}{i}{j} & \quad -\frac{1}{2} \leq j \leq l-1, \ \frac{1}{2} \leq l \leq k \\
0 & \quad \mbox{otherwise}.
\end{cases}
}

So if we set $n = 2j \pm 1$ and recall the sets

\enveqn{
\mathcal{J}_{l} := \begin{cases}
\{ 0, 2, \ldots, 2l-1 \} & \quad l \in (\bN_{0} + \thalf) \\
\{ 1, 3, \ldots, 2l-1 \} & \quad l \in \bN
\end{cases}
}

we may express the trace as

\enval{
& \mathrm{Tr} \left( P_{1} P_{2} L_{k} (1 + \cD^{2})^{-s/4} \hat\Delta_{F}^{-\half} \omega \hat\Delta_{F}^{-\half} (1 + \cD^{2})^{-s/4}  P_{1} P_{2} L_{k}\right) = \nonumber \\
&= \sum_{2l = 1}^{2k} \sum_{i=-l}^{l} \sum_{n \in \mathcal{J}_{l}} \frac{q^{-2i-n}}{(1 + \lambda_{l, n}^{2})^{s/2}} \left( \frac{\la  \re{l}{i}{\frac{n+1}{2}}, \omega^+\re{l}{i}{\frac{n+1}{2}} \ra}{\la \re{l}{i}{\frac{n+1}{2}}, \re{l}{i}{\frac{n+1}{2}} \ra} + \frac{\la \re{l}{i}{\frac{n-1}{2}}, \omega^- \re{l}{i}{\frac{n-1}{2}} \ra}{\la \re{l}{i}{\frac{n-1}{2}}, \re{l}{i}{\frac{n-1}{2}} \ra} \right). \label{eqn_trace_formula}
}

This  shows that $\Upsilon_s$ is a supremum of a sum of positive vector states and so
automatically positive and normal. To see that it is faithful on $P_1P_2\B(\cH)_+P_1P_2$
we observe that the operator trace is faithful and that $P_1P_2\hat\Delta^{-1/2}_F(1+\D^2)^{-s/4}$ is injective
on $P_1P_2\cH$.
 The semifiniteness comes from the fact that finite rank operators 
are in the domain of $\Upsilon_s$.
\end{proof}

We extend $\Upsilon_s$ to an unbounded positive
normal linear functional on $\B(\cH)$ as
usual. In fact, we extend it also to
unbounded operators $\omega$ defined on and
preserving $\A \oplus \A$ by decomposing
$L_k \omega L_k$ for each $k$ into a linear combination of
positive bounded operators. 

If for
an operator $\omega$ (not necessarily bounded) the function $s\mapsto \Upsilon_s(\omega)$
has a meromorphic continuation to $\domain_{3-\delta}$ for some $\delta>0$, then we define

\enveqn{
\tau(\omega) := 
\mathrm{Res}_{z=3} \Upsilon_{z} (\omega). }


\begin{lemma}
\label{lemma_tau_bc_zero}

The functional $\tau$ is defined on the positive operator $c^{\ast}c$, and $\tau(c^*c)=0$. 
Indeed, for all $m \geq 1$,

\enveqn{
\tau \left( \bma (c^{\ast}c)^{m} & 0 \\ 0 & 0 \ema \right) = \tau \left( \bma 0 & 0 \\ 0 & (c^{\ast}c)^{m} \ema  \right) = 0.
}

\end{lemma}
\begin{proof}

The action of the operator $c = c^{+} + c^{-}$ may be described using
the Clebsch-Gordan coefficients (see for example \cite{DLSSV},
\cite{KS}): we have 

\envaln{
c^{+} \re{l}{i}{j} &= c^{l+}_{ij} \re{l+\half}{i+\half}{j-\half} &
c^{-} \re{l}{i}{j} &= c^{l-}_{ij} \re{l-\half}{i+\half}{j-\half},
}

where 

\envaln{
c^{l+}_{ij} &= q^{(i+j)/2} \frac{\left( [l+i+1]_q [l-j+1]_q \right)^{1/2}}{[2l+1]_q}, &
c^{l-}_{ij} &= - q^{(i+j)/2} \frac{\left( [l-i]_q[l+j]_q \right)^{1/2}}{[2l+1]_q}.
}

Using this description of $c$ to compute the action of $c^{\ast} c$, we find

\envaln{
 (c^{\ast}c) \re{l}{i}{j} 
&= q^{i+j-1} \left( \frac{[l+i+1]_q[l-j+1]_q}{[2l+1]_q[2l+2]_q} + \frac{[l-i]_q[l+j]_q}{[2l]_q[2l+1]_q} \right) \re{l}{i}{j} \\
& \quad- q^{i+j-1} \left(
 \frac{([l+i+1]_q[l-i+1]_q[l+j+1]_q[l-j+1]_q)^{\half}}{[2l+1]_q[2l+2]_q}
 \re{l+1}{i}{j}\right. \\
&\left.\quad+ \frac{([l+i]_q[l-i]_q[l+j]_q[l-j]_q)^{\half}}{[2l]_q[2l+1]_q} \re{l-1}{i}{j} \right)
}

Let $\epsilon_{k} = Q(1-q^{2k})$, so that $[k]_{q} = q^{-k} \epsilon_{k}$. Then the 
above expression can be written as

\envaln{
 (c^{\ast}c) \re{l}{i}{j} &= q^{2l} \left( q^{2j} \frac{\epsilon_{l+i+1}\epsilon_{l-j+1}}{\epsilon_{2l+1} \epsilon_{2l+2}} + q^{2i} \frac{\epsilon_{l-i} \epsilon_{l+j}}{\epsilon_{2l} \epsilon_{2l+1}} \right) \re{l}{i}{j} \\
& - q^{2l+i+j} \left( \frac{(\epsilon_{l+i+1} \epsilon_{l-i+1} \epsilon_{l+j+1} \epsilon_{l-j+1})^{\half}}{\epsilon_{2l+1} \epsilon_{2l+2}} \re{l+1}{i}{j} + \frac{(\epsilon_{l+i} \epsilon_{l-i} \epsilon_{l+j} \epsilon_{l-j})^{\half}}{\epsilon_{2l} \epsilon_{2l+1}} \re{l-1}{i}{j} \right).
}

Define the scalars $C_{1}(l, i, j)$ and $C_{2}(l, i, j)$ to be

\envaln{
C_{1}(l, i, j) &:= \frac{\epsilon_{l+i+1}\epsilon_{l-j+1}}{\epsilon_{2l+1} \epsilon_{2l+2}} &
C_{2}(l, i, j) &:= \frac{\epsilon_{l-i} \epsilon_{l+j}}{\epsilon_{2l} \epsilon_{2l+1}}.
}

The definition of $\epsilon_{k}$ implies that $C_{1}$ and $C_{2}$ are uniformly 
bounded for all  $l,\,i,\,j$ appearing in the formula for $\Upsilon_z(c^*c)$.
%

As in the proof of Lemma \ref{lemma_operator_trace} we compute for $z\in \mathbb{R}$

\envaln{
& \mathrm{Tr} \left( P_{1} P_{2} L_{k}
 (1 + \cD^{2})^{-z/4}
 \hat\Delta_{F}^{-\half} 
\bma c^{\ast}c & 0 \\ 0 & 0 \ema  \hat\Delta_{F}^{-\half} (1 + \cD^{2})^{-z/4} P_{1} P_{2} L_{k} \right) \\
&\qquad\qquad = \sum_{2l = 1}^{2k} \sum_{i=-l}^{l} \sum_{n \in \mathcal{J}_{l}} \frac{q^{-2i-n}}{(1 + \lambda_{l, n}^{2})^{z/2}} \left( q^{2l+n+1}C_{1}(l, i, \tfrac{n+1}{2}) + q^{2l+2i}C_{2}(l, i, \tfrac{n+1}{2}) \right),
}

\envaln{
& \mathrm{Tr} \left( P_{1} P_{2} L_{k}
 (1 + \cD^{2})^{-z/4}
 \hat\Delta_{F}^{-\half} 
\bma 0 & 0 \\ 0 & c^{\ast}c \ema  \hat\Delta_{F}^{-\half} (1 + \cD^{2})^{-z/4} P_{1} P_{2} L_{k} \right) \\
&\qquad\qquad = \sum_{2l = 1}^{2k} \sum_{i=-l}^{l} \sum_{n \in \mathcal{J}_{l}} \frac{q^{-2i-n}}{(1 + \lambda_{l, n}^{2})^{z/2}} \left( q^{2l+n+1}C_{1}(l, i, \tfrac{n-1}{2}) + q^{2l+2i}C_{2}(l, i, \tfrac{n-1}{2}) \right).
}

The uniform boundedness of $C_{1}$ and $C_{2}$, together with Lemma \ref{lemma_holomorphic}, 
demonstrate that the limits as $k \rightarrow \infty$ of the two sums above exist for $z>2$. Hence
$$
z\mapsto \Upsilon_z\left(\bma c^{\ast}c & 0 \\ 0 & 0 \ema\right),\quad 
z\mapsto \Upsilon_z\left(\bma 0 & 0 \\ 0 &  c^{\ast}c  \ema\right)
$$
are well-defined functions for $z>2$.
Indeed the arguments of Lemma \ref{lemma_holomorphic}, together with the Weierstrass convergence 
theorem, show that these functions extend to holomorphic functions  
on $\domain_2$. 
In particular, these functions are holomorphic  at $z = 3$ and hence

\enveqn{
\tau \left( \bma c^{\ast}c & 0 \\ 0 & 0 \ema \right) = \tau \left( \bma 0 & 0 \\ 0 & c^{\ast}c \ema  \right) = 0.
}

By linearity it follows that $\tau(c^{\ast}c) = 0$ also. Using the normality of $c$, for any
operator $X$ we have the operator inequality
$$
X^*(c^*c)^mX\leq \Vert c^*c\Vert ^{m-1} X^*c^*cX,
$$
and so for $z>2$ real, we have $\Upsilon_{z}((c^{\ast}c)^{m}) \leq \Vert c \Vert^{2m-2} \Upsilon_{z}(c^{\ast}c)$.
Thus for $z>2$, the sum defining $\Upsilon_{z}((c^{\ast}c)^{m})$ converges. Once more invoking
the Weierstrass convergence theorem shows that $z\mapsto \Upsilon_{z}((c^{\ast}c)^{m})$ extends to
a holomorphic function for ${\mathrm Re}(z)>2$. Similar estimates now show that

\enveqn{
\tau \left( \bma (c^{\ast}c)^{m} & 0 \\ 0 & 0 \ema \right) 
= \tau \left( \bma 0 & 0 \\ 0 & (c^{\ast}c)^{m} \ema  \right) = 0. \qedhere
}
\end{proof}

%



\begin{thm}
\label{thm_res_gives_int}
Let $\alpha \in \A$ and $X, Y$ be any closed linear operators 
 on $\cH_{h}$ which are defined on
 and preserve $\A$. Then we have the following well-defined evaluations of $\tau$: 

\begin{enumerate}
\item $\tau \left( \bma 0 & X \\ 0 & 0 \ema \right) = \tau \left( \bma 0 & 0 \\ Y & 0 \ema \right) = 0$
\item $\tau(\alpha \Gamma) = 0$
\item $\tau \left(\hat{\Delta}_{L}^{2}
		\bma \alpha & 0 \\ 0 & 0 \ema
		\right) =
		\tau\left(\hat{\Delta}_{L}^{2} \bma
		0 & 0 \\ 0 & \alpha \ema \right) = R \int_{[1]} \alpha$
\end{enumerate}

where $\int_{[1]} \colon \A \rightarrow \bC$ is the functional defined in Lemma \ref{haddock}
and $R = 4(q^{-1}-q)/\ln(q^{-1})$.

\end{thm}
\begin{proof}

Throughout this proof we assume without loss of 
generality that any element of $\A$ is homogeneous with respect to both the left and right actions 
(that is $\sigma_{L}(\alpha) = q^{p} \alpha$, $\sigma_{R}(\alpha) = q^{p'} \alpha$ for some $p, p'$). This is because finite linear combinations of homogeneous elements span $\cA$ (cf. Theorem \ref{thm_rep_basis}).

Indeed, if $\alpha \in \cA$ is homogeneous of a non-zero degree for either the left or right action,
then $\langle \re{l}{i}{j},\alpha\, \re{l}{i}{j}\rangle=0$ and so for any linear operator 
$C$ that is diagonal in the
Peter-Weyl basis, $\Upsilon_{s}(C \alpha) = 0$ for all $s \in \bR_{+}$. Hence, 
we need only consider those elements of $\cA$ that are homogeneous of 
degree zero for the left and right actions.
A convenient spanning set for these algebra elements is $\{ 1_\A, (c^{\ast}c)^{m} \colon m \in \bN \}$.

{\em 1.} By definition 
$\Upsilon_s\left(\left(\begin{array}{cc}
0  & X \\
0 & 0
\end{array}\right)  \right)=0$ for all
 $s>0$, and similarly for
		  $\left(\begin{array}{cc} 0 &0 \\Y &0
					\end{array}\right)  $.

{\em 2.} Lemma \ref{lemma_tau_bc_zero} has established that for all $m \geq 1$,

\enveqn{
\tau \left( \bma (c^{\ast}c)^{m} & 0 \\ 0 & 0 \ema  \right) = \tau \left( \bma 0 & 0 \\ 0 & (c^{\ast}c)^{m} \ema \right) = 0.
}

By linearity we can extend this to conclude that $\tau((c^{\ast}c)^{m} \Gamma) = 0$. 
Finally, for $z$ large and real we compute $\Upsilon_{z}(\Gamma)$ using the 
proof of Lemma \ref{lemma_operator_trace}. Now

\envaln{
& \mathrm{Tr} \left( P_{1} P_{2} L_{k} (1 + \cD^{2})^{-z/4} \hat\Delta_{F}^{-\half} \Gamma \hat\Delta_{F}^{-\half} (1 + \cD^{2})^{-z/4} P_{1} P_{2} L_{k} \right) \\
&\qquad = \sum_{2l=1}^{2k} \sum_{i=-l}^{l} \sum_{n \in \mathcal{J}_{l}} \frac{q^{-2i-n} }{ (1 + \lambda_{l, n}^{2})^{z/2} } - \sum_{2l=1}^{2k} \sum_{i=-l}^{l} \sum_{n \in \mathcal{J}_{l}} \frac{q^{-2i-n} }{ (1 + \lambda_{l, n}^{2})^{z/2} },
}

and for each $k$ the summands above are finite and hence subtract to give zero.
Hence $\Upsilon_{z}(\Gamma) = 0$ for all $z$ and so $\tau(\Gamma) = 0$.

{\em 3.} For $z$ large and real, the evaluation of $\Upsilon_{z}$ as sums of 
positive real numbers (as in the proof of Lemma \ref{lemma_operator_trace}) 
implies the numerical inequality

\enveqn{
\Upsilon_{z}(\hat{\Delta}_{L}^{2} (c^{\ast}c)^{m}) \leq \Upsilon_{z}((c^{\ast}c)^{m}).
}

This is because the introduction of $\hat{\Delta}_{L}^{2}$ multiplies each 
summand by $q^{2n} \leq 1$ (cf. Equation \eqref{eqn_trace_formula}). 
Lemma \ref{lemma_tau_bc_zero} demonstrates that $\Upsilon_{z}((c^{\ast}c)^{m})$ 
extends to a function that is holomorphic in a neighbourhood of 
$z = 3$, and together with the Weierstrass 
convergence theorem the result follows.

Finally we analyse $\Upsilon_{z}\left(\hat\Delta_{L}^{2} \bma 1 & 0 \\ 0 & 0 \ema \right)$ 
and $\Upsilon_{z}\left(\hat\Delta_{L}^{2} \bma 0 & 0 \\ 0 & 1 \ema \right)$. Again 
using the proof of Lemma \ref{lemma_operator_trace} we find 

\envaln{
& \mathrm{Tr} \left( P_{1} P_{2} L_{k} (1 + \cD^{2})^{-z/4} \hat\Delta_{F}^{-\half} \hat\Delta_{L}^{2} \bma 1 & 0 \\ 0 & 0 \ema \hat\Delta_{F}^{-\half} (1 + \cD^{2})^{-z/4} P_{1} P_{2} L_{k} \right) \\
& \qquad = \mathrm{Tr} \left( P_{1} P_{2} L_{k} (1 + \cD^{2})^{-z/4} \hat\Delta_{F}^{-\half} \hat\Delta_{L}^{2} \bma 0 & 0 \\ 0 & 1 \ema \hat\Delta_{F}^{-\half} (1 + \cD^{2})^{-z/4} P_{1} P_{2} L_{k} \right) \\
& \qquad\qquad = \sum_{2l=1}^{2k} \sum_{i=-l}^{l} \sum_{n \in \mathcal{J}_{l}} \frac{q^{-2i+n} }{ (1 + \lambda_{l, n}^{2})^{z/2} } \\
& \qquad\qquad\qquad = Qq^{-1} \sum_{2l=1}^{2k} \sum_{n \in \mathcal{J}_{l}} \frac{q^{n-2l} }{ (1 + \lambda_{l, n}^{2})^{z/2} } - Q q\sum_{2l=1}^{2k} \sum_{n \in \mathcal{J}_{l}} \frac{q^{n+2l} }{ (1 + \lambda_{l, n}^{2})^{z/2} }
}

For $z$ real, the sum $\sum_{2l=1}^{2k} \sum_{n \in \mathcal{J}_{l}} q^{n+2l}/(1 + \lambda_{l, n}^{2})^{z/2}$ is bounded above by $f_{2}(z)$ from Lemma \ref{lemma_holomorphic} for all $k$. By the Weierstrass convergence theorem we conclude that as $k \rightarrow \infty$, this sum converges to a function with a holomorphic extension about $z = 3$. Next, when considering the sum $\sum_{2l=1}^{2k} \sum_{n \in \mathcal{J}_{l}} q^{n-2l}/(1 + \lambda_{l, n}^{2})^{z/2}$, observe by rearranging the order of summation
$$
\sum_{2l = 1}^{2k} \sum_{n \in \mathcal{J}_{l}} 
\rightarrow \sum_{n = 0}^{2k} \sum_{l = (n+1)/2}^{k},
$$ 

that Lemma \ref{lemma_tau_residue} proves that the sum has a limit as 
$k \rightarrow \infty$ and the corresponding function of $z$ extends to a meromorphic function
with a simple pole at $z = 3$. The residue at $z = 3$ is $4qQ^{-2}/\ln(q^{-1})$ 
and from the definition of $\tau$ we conclude that for $R = 4(q^{-1}-q)/\ln(q^{-1})$,

\enveqn{
\tau\left(\hat\Delta_{L}^{2} \bma 1 & 0 \\ 0 & 0 \ema \right) = \tau\left(\hat\Delta_{L}^{2} \bma 0 & 0 \\ 0 & 1 \ema \right) = R.
}

Finally, we compare the definition of $R \int_{[1]}$ in Lemma \ref{haddock} 
to the evaluation of $\tau$ on $\A$ derived here
and observe that they agree on $\A$.
\end{proof}

\begin{lemma}
\label{lemma_twisted_trace}

Given any matrix $M \in \M_{2}(\cA)$ and any $\alpha \in \cA$ then $\tau(M \hat{\Delta}_{L}^{2} \alpha) = \tau(\vartheta^{-1}(\alpha) M \hat{\Delta}_{L}^{2})$.

\end{lemma}
\begin{proof}

From Lemma \ref{haddock}, the linear functional $\int_{[1]}$ is a $\sigma_{L}^{2} \circ \vartheta^{-1}$-twisted trace. That is, given any $\alpha, \beta \in \cA$

\enveqn{
\int_{[1]} \alpha \beta = \int_{[1]} \sigma_{L}^{2}(\vartheta^{-1}(\beta)) \alpha.
}

Now we separate the matrix $M = M_{d} + M_{o}$ into diagonal and off-diagonal matrices respectively. 
Then by Theorem \ref{thm_res_gives_int}, 
$\tau(M_{d} \hat{\Delta}_{L}^{2} \alpha)$ and $\tau(M_{o} \hat{\Delta}_{L}^{2} \alpha)$ 
are both well-defined, so by linearity

\enveqn{
\tau(M \hat{\Delta}_{L}^{2} \alpha) = \tau(M_{d} \hat{\Delta}_{L}^{2} \alpha) + \tau(M_{o} \hat{\Delta}_{L}^{2} \alpha) = \tau(M_{d} \hat{\Delta}_{L}^{2} \alpha) + 0.
}

Since $M_{d}$ is diagonal, we may write

\enveqn{
M_{d} \hat{\Delta}_{L}^{2} = \hat{\Delta}_{L}^{2} \sigma_{L}^{2}(M_{d})
}

where $\sigma_{L}$ acts componentwise on the matrix. Using the value of 
$\tau(\hat{\Delta}_{L}^{2} \sigma_{L}^{2}(M_{d}) \alpha)$ from Theorem \ref{thm_res_gives_int}, we have

\envaln{
\tau(M \hat{\Delta}_{L}^{2} \alpha) = \tau(\hat{\Delta}_{L}^{2} \sigma_{L}^{2}(M_{d}) \alpha) &= R \int_{[1]} \sigma_{L}^{2}(M_{d}^{+}) \alpha + R \int_{[1]} \sigma_{L}^{2}(M_{d}^{-}) \alpha \\
&= R \int_{[1]} \sigma_{L}^{2}(\vartheta^{-1}(\alpha)) \sigma_{L}^{2}(M_{d}^{+}) + R \int_{[1]} \sigma_{L}^{2}(\vartheta^{-1}(\alpha)) \sigma_{L}^{2}(M_{d}^{-}),
}

by the twisted trace property of $\int_{[1]}$. Recombining these two terms yields
%

\envaln{
\tau(\hat{\Delta}_{L}^{2} \sigma_{L}^{2}(M_{d}) \alpha)
= \tau(\hat{\Delta}_{L}^{2} \sigma_{L}^{2}(\vartheta^{-1}(\alpha)) \sigma_{L}^{2}(M_{d}))
= \tau(\vartheta^{-1}(\alpha) \hat{\Delta}_{L}^{2} \sigma_{L}^{2}(M_{d}))
= \tau(\vartheta^{-1}(\alpha) M_{d} \hat{\Delta}_{L}^{2}).
}

Now, $\tau(\vartheta^{-1}(\alpha) M_{o} \hat{\Delta}_{L}^{2})$ is well defined and has value zero, so we can write

\enveqn{
\tau(M \hat{\Delta}_{L}^{2} \alpha) = \tau(\vartheta^{-1}(\alpha) M_{d} \hat{\Delta}_{L}^{2}) + \tau(\vartheta^{-1}(\alpha) M_{o} \hat{\Delta}_{L}^{2}) = \tau(\vartheta^{-1}(\alpha) M \hat{\Delta}_{L}^{2}). \qedhere
}
\end{proof}

\begin{thm}
\label{thm:big-fish}
Given any $a_{0}, \ldots, a_{3} \in \cA$, the map 
$\phi_{\mathrm{res}}:a_{0}, \ldots, a_{3} \mapsto \tau(a_{0} [\cD, a_{1}] [\cD, a_{2}] [\cD, a_{3}])$ is a $\vartheta^{-1}$-twisted Hochschild 3-cocycle, whose cohomology class is
non-trivial. The cocycle $\phi_{\mathrm{res}}$ has non-zero 
pairing with the $\vartheta^{-1}$-twisted 3-cycle $dvol$ defined in \eqref{eqn_dvol}, giving

\enveqn{
\la \phi_{\mathrm{res}}, dvol \ra = 3R(q^{-1} + q)=4!\,\frac{q^{-1}+q}{2}\,\frac{q^{-1}-q}{\ln(q^{-1})}.
}
The cocycle may be written as

\enveqn{
\phi_{\mathrm{res}} = q^{2}R(\varphi + \varphi_{213} + \varphi_{231}) + R(\varphi_{132} + \varphi_{312} + \varphi_{321})
}

where $\varphi$ and $\varphi_{ijk}$ are the cocycles described in Lemma \ref{haddock} 
and Corollary \ref{cor:perms}.

\end{thm}
\begin{proof}

First consider $\pi_{\cD}(a_{0}, a_{1}, a_{2}, a_{3}) = a_{0} [\cD, a_{1}] [\cD, a_{2}] [\cD, a_{3}]$ 
as an unbounded operator on $\A\oplus\A\subset\cH$. Using the equality 
$[\cD, \alpha] = \coms(\alpha) + \comt(\alpha)\hat\Delta_{L}$, we see that 
$\pi_{\cD}(a_{0}, a_{1}, a_{2}, a_{3})$ can be expanded into 8 terms. 
Recall that by Theorem \ref{thm_res_gives_int} the functional $\tau$ vanishes on 
off-diagonal operators. Four of the eight terms in the expansion of 
$\pi_{\cD}(a_{0}, a_{1}, a_{2}, a_{3})$ are off-diagonal since, for all $\alpha\in\A$, $\coms(\alpha)$ is 
diagonal and $\comt(\alpha)$ is off-diagonal. Thus

\envaln{
& \tau \left( a_{0} \left( \comt(a_{1})\hat\Delta_{L} \coms(a_{2}) \coms(a_{3}) + \coms(a_{1}) \comt(a_{2})\hat\Delta_{L} \coms(a_{3}) \right. \right. \\
& \quad \left. \left. + \coms(a_{1}) \coms(a_{2}) \comt(a_{3})\hat\Delta_{L} + \comt(a_{1}) \hat\Delta_{L} \comt(a_{2}) \hat\Delta_{L} \comt(a_{3}) \hat\Delta_{L} \right) \right) = 0.
}

Therefore,  $\phi_{\mathrm{res}}(a_{0},a_{1}, a_{2}, a_{3})$ reduces to 

\enval{
\phi_{\mathrm{res}}(a_{0},  a_{1}, a_{2}, a_{3})
&= \tau \left( a_{0} \left( \coms(a_{1}) \coms(a_{2}) \coms(a_{3}) 
+ \coms(a_{1}) \comt(a_{2}) \hat\Delta_{L} \comt(a_{3}) \hat\Delta_{L} \right. \right. \nonumber \\
& \quad \left. \left. + \comt(a_{1}) \hat\Delta_{L} \coms(a_{2}) \comt(a_{3}) \hat\Delta_{L} 
+ \comt(a_{1}) \hat\Delta_{L} \comt(a_{2}) \hat\Delta_{L} \coms(a_{3}) \right) \right).
\label{eq:three-not-four}
}

From Lemma \ref{lemma_triple_properties} it follows that

\enveqn{
a_{0} \coms(a_{1}) \coms(a_{2}) \coms(a_{3}) 
= a_{0} \partial_{H}(a_{1}) \partial_{H}(a_{2}) \partial_{H}(a_{3}) \Gamma
}

and recall that from Theorem \ref{thm_res_gives_int}, $\tau(\alpha \Gamma) = 0$ 
for all $\alpha \in \A$. Since 
$a_{0} \partial_{H}(a_{1}) \partial_{H}(a_{2}) \partial_{H}(a_{3}) \in \A$ we have

\enveqn{
\tau (a_{0} \coms(a_{1}) \coms(a_{2}) \coms(a_{3})) = 0.
}

We now move all the $\hat{\Delta}_L$'s to the right in the remaining terms in Equation \eqref{eq:three-not-four}. 
For $\alpha\in\A$, we use 
 $\hat\Delta_{L} \coms(\alpha) = \coms(\sigma_{L}^{-1}(\alpha)) \hat\Delta_{L}$ ,  and
 
\envaln{
\hat\Delta_{L} \comt(\alpha) 
&= \bma 0 & q^{-1} \Delta_{L} q^{-\half} \pe(\sigma_{L}^{-\half}(\alpha)) \\ 
q \Delta_{L} q^{\half} \pf(\sigma_{L}^{-\half}(\alpha)) & 0 \ema \\
&= \bma 0 & q^{-2} q^{-\half} \sigma_{L}^{-1}(\pe(\sigma_{L}^{-\half}(\alpha))) \\ 
q^{2} q^{\half} \sigma_{L}^{-1}(\pf(\sigma_{L}^{-\half}(\alpha))) & 0 \ema \hat{\Delta}_{L} \\
&= \bma 0 & q^{-\half} \pe(\sigma_{L}^{-\frac{3}{2}}(\alpha)) \\ 
q^{\half} \pf(\sigma_{L}^{-\frac{3}{2}}(\alpha)) & 0 \ema \hat{\Delta}_{L} \\
&= \comt(\sigma_{L}^{-1}(\alpha)) \hat{\Delta}_{L}.
}

This yields

\envaln{
\phi_{\mathrm{res}}(a_{0}, \,& a_{1}, a_{2}, a_{3}) 
= \tau \left( a_{0} \left( \coms(a_{1}) \comt(a_{2}) \comt(\sigma_{L}^{-1}(a_{3})) \right. \right. \\
& \left. \left. + \comt(a_{1}) \coms(\sigma_{L}^{-1}(a_{2})) \comt(\sigma_{L}^{-1}(a_{3})) 
+ \comt(a_{1}) \comt(\sigma_{L}^{-1}(a_{2})) \coms(\sigma_{L}^{-2}(a_{3})) \right) \hat{\Delta}_{L}^{2} \right).
}

In this form Theorem \ref{thm_res_gives_int} tells us that $\phi_{\mathrm{res}}$ 
is a well defined, multilinear functional on $\cA^{\ox 4}$. 
In order to demonstrate that this cochain is indeed a twisted Hochschild cocycle, 
it remains only to show that the boundary operator maps the cochain to zero. 
This result follows from the Leibniz property of the commutators together with 
Lemma \ref{lemma_twisted_trace}. Explicitly, 

\envaln{
& (b_{3}^{\vartheta^{-1}} \phi_{\mathrm{res}}) (a_{0}, \ldots, a_{4}) = \tau(a_{0}a_{1} [\cD, a_{2}] [\cD, a_{3}] [\cD, a_{4}]) - \tau(a_{0} [\cD, a_{1}a_{2}] [\cD, a_{3}] [\cD, a_{4}]) \\
& \qquad + \tau(a_{0} [\cD, a_{1}] [\cD, a_{2}a_{3}] [\cD, a_{4}]) - \tau(a_{0} [\cD, a_{1}] [\cD, a_{2}] [\cD, a_{3}a_{4}]) + \tau(\vartheta^{-1}(a_{4})a_{0} [\cD, a_{1}] [\cD, a_{2}] [\cD, a_{3}]) \\
& \qquad \qquad = - \tau(a_{0} [\cD, a_{1}] [\cD, a_{2}] [\cD, a_{3}]a_{4}) + \tau(\vartheta^{-1}(a_{4})a_{0} [\cD, a_{1}] [\cD, a_{2}] [\cD, a_{3}]) = 0,
}
where the last equality follows from Lemma \ref{lemma_twisted_trace}.
In order to identify $\phi_{\mathrm{res}}$, we use Lemma \ref{lemma_triple_properties} to write, 
for $a_{0}, \ldots, a_{3} \in \A$,  

\envaln{
a_{0} & \left( \coms(a_{1}) \comt(a_{2}) \comt(\sigma_{L}^{-1}(a_{3})) 
+ \comt(a_{1}) \coms(\sigma_{L}^{-1}(a_{2})) \comt(\sigma_{L}^{-1}(a_{3})) \right. \\
& \left. + \comt(a_{1}) \comt(\sigma_{L}^{-1}(a_{2})) \coms(\sigma_{L}^{-2}(a_{3})) \right) 
= \bma \pi_{1}(a_{0}, \ldots, a_{3}) & 0 \\ 0 & \pi_{2}(a_{0}, \ldots, a_{3}) \ema
}

for some multi-linear maps $\pi_{1}, \pi_{2} \colon \A^{\ox 4} \rightarrow \A$. 
Again using Lemma \ref{lemma_triple_properties}, we have

\enval{
\pi_{1}(a_{0}, \ldots, a_{3}) 
&= a_{0} \partial_{H}(a_{1}) \pe(\sigma_{L}^{-\half}(a_{2})) \pf(\sigma_{L}^{-\frac{3}{2}}(a_{3}))  - a_{0} \pe(\sigma_{L}^{-\half}(a_{1})) \partial_{H}(\sigma_{L}^{-1}(a_{2})) 
\pf(\sigma_{L}^{-\frac{3}{2}}(a_{3})) \nonumber \\
& \quad + a_{0} \pe(\sigma_{L}^{-\half}(a_{1})) \pf(\sigma_{L}^{-\frac{3}{2}}(a_{2})) 
\partial_{H}(\sigma_{L}^{-2}(a_{3})), \label{eqn_pi_one}
}

\enval{
\pi_{2}(a_{0}, \ldots, a_{3}) 
&= - a_{0} \partial_{H}(a_{1}) \pf(\sigma_{L}^{-\half}(a_{2})) \pe(\sigma_{L}^{-\frac{3}{2}}(a_{3}))  + a_{0} \pf(\sigma_{L}^{-\half}(a_{1})) \partial_{H}(\sigma_{L}^{-1}(a_{2})) 
\pe(\sigma_{L}^{-\frac{3}{2}}(a_{3})) \nonumber \\
& \quad - a_{0} \pf(\sigma_{L}^{-\half}(a_{1})) \pe(\sigma_{L}^{-\frac{3}{2}}(a_{2})) 
\partial_{H}(\sigma_{L}^{-2}(a_{3})). \label{eqn_pi_two}
}

Then by Theorem \ref{thm_res_gives_int}, and the $\sigma_L$ invariance of $\int_{[1]}$, we have

\enveq{
\phi_{\mathrm{res}}(a_{0} , a_{1}, a_{2}, a_{3}) 
= R \int_{[1]} \pi_{1}(a_{0}, \ldots, a_{3}) + R \int_{[1]} \pi_{2}(a_{0}, \ldots, a_{3}). \label{eqn_cocycle_pi}
}

Comparing Equations \eqref{eqn_pi_one}, \eqref{eqn_pi_two}, \eqref{eqn_cocycle_pi} with the expressions for the cocycles identified in Lemma \ref{haddock} and Corollary \ref{cor:perms} we find

\enveqn{
\phi_{\mathrm{res}} = 
q^{2}R (\varphi + \varphi_{213} + \varphi_{231}) + R (\varphi_{132} + \varphi_{312} + \varphi_{321}).
}

The evaluation of this cocycle on the cycle $dvol$ (see Equation \eqref{eqn_dvol}) is a straightforward computation using the explicit expressions obtained. The result is

\enveqn{
\la \phi_{\mathrm{res}} , dvol \ra = 3R(q^{-1} + q). \qedhere
}
\end{proof}



%

\end{document}